\documentclass[12pt]{amsart}

\usepackage{amssymb}

\usepackage{enumerate}

\usepackage{graphicx}

\makeatletter
\@namedef{subjclassname@2010}{%
  \textup{2010} Mathematics Subject Classification}
\makeatother

\theoremstyle{definition}

\numberwithin{equation}{section}

\usepackage[all]{xy}
\usepackage{mathrsfs}

\newtheorem{theorem}{Theorem}[section]
\newtheorem{lemma}[theorem]{Lemma}
\newtheorem{proposition}[theorem]{Proposition}

\newcommand {\OK} {\mathcal O_K}
\newcommand {\Ov}  {{\mathcal O_{v}}}
\newcommand{\ord}{\mathrm{ord}}
\newcommand{\Lie}{\mathrm{Lie}}
\newcommand{\id}{\mathrm{id}}
\newcommand{\den}{\mathrm{d}}
\newcommand{\G}{\mathscr G}
\newcommand{\W}{\mathscr W}
\newcommand{\HH}{\mathscr H}
\newcommand{\LL}{\mathscr L}
\newcommand{\uu}{\mathbf{u}_0}
\newcommand{\w}{\mathbf{u}}
\newcommand{\dd}{{\delta_L}}
\newcommand{\ga}{\gamma}
\newcommand{\diff}{\Delta}
\newcommand{\codim}{\mathrm{codim}}
\newcommand{\card}{\mathrm{card}}

\newcommand{\Dl}{L}
\newcommand{\p}{\partial}
\newcommand{\n}{{n(u)}}
\newcommand{\hm}{{h_{\max}}}
\newcommand{\hp}{{h^+_{\max}}}
\newcommand{\ccc}{{d_L}}
\newcommand{\cd}{{\delta_L}}
\newcommand{\cc}{{h_L}}
\newcommand{\dis}{{\rm disc}}
\newcommand{\Exp}{{\mathrm {Exp}}}
\newcommand{\oL}{{\omega_L}}

\begin{document}


\baselineskip=17pt



\title[Commutative algebraic groups and $p$-adic linear forms]{Commutative algebraic groups and $p$-adic linear forms}

\author[C. Fuchs]{Clemens Fuchs}
\address{Department of Mathematics\\ University of Salzburg\\ Hellbrunnerstr. 34\\ 5020 Salzburg, Austria}
\email{clemens.fuchs@sbg.ac.at}

\author[D. H. Pham]{Duc Hiep Pham}
\address{Department of Mathematics\\ Hanoi National University of Education\\ 136 Xuan Thuy, Cau Giay, Hanoi, Vietnam}
\email{phamduchiepk6@gmail.com}

\date{}

\begin{abstract}
Let $G$ be a commutative algebraic group defined over a number field $K$ that is disjoint over $K$ to $\mathbb G_a$ and satisfies the condition of semistability. Consider a linear form $l$ on the Lie algebra of $G$ with algebraic coefficients and an algebraic point $u$ in a $p$-adic neighbourhood of the origin with the condition that $l$ does not vanish at $u$. We give a lower bound for the $p$-adic absolute value of $l(u)$ which depends up to an effectively computable constant only on the height of the linear form, the height of the point $u$ and $p$.
\end{abstract}



\maketitle

\section{Introduction}

The theory of Diophantine approximation is one of the most interesting problems in number theory in which the theory of linear forms plays a central role. In 1966 Baker made a breakthrough by proving a very deep result on effective lower bounds for linear forms in logarithms of algebraic numbers (see the series of papers \cite{ba1}). This result was refined by Baker and W\"ustholz (see \cite{bw10}). After W\"ustholz proved a brilliant theorem which is called the analytic subgroup theorem (see \cite{aw} or \cite{w1}), the problem of linear forms could be considered in higher dimensions. In the literature one can find generalizations in terms of algebraic groups and the most general results so far are due to Hirata-Kohno (see \cite{kohno}) and Gaudron (see \cite{g}).

It is natural to consider $p$-adic analogues of such problems. The theory of $p$-adic linear forms plays indeed an important and fundamental role in number theory. It has been applied to many questions, in particular it was successfully used to solve completely a large number of Diophantine problems of different shape. One of the interests comes from the problem of finding lower bounds for linear forms in $p$-adic logarithm functions evaluated at algebraic points. Unlike in the complex case, the $p$-adic logarithm function is only defined locally. It is therefore more natural to study upper bounds for the $p$-adic valuation of expressions $\alpha_1^{b_1}\cdots\alpha_n^{b_n}-1$ where $\alpha_1,\ldots,\alpha_n$ are algebraic numbers such that they are multiplicatively independent and $b_1,\ldots,b_n$ are rational integers, not all zero. Such problems have been investigated by many authors (see e.g. \cite{caa}) and the most outstanding results to date are due to Yu (see \cite{kr1,kr2,kr3,kr4}). In 1998 he formulated and proved a $p$-adic analogue of the Baker and W\"ustholz theorem and afterwards in a series of papers he improved the bounds. The results of Yu were used by Stewart and himself to deal with the $abc$-conjecture (see \cite{ty1}). In particular, Stewart and Yu in 2001 showed that there is an effectively computable positive number $c$ such that for all coprime positive integers $x, y$ and $z>2$ with $x+y=z$ one has \[z<\exp\Big(c N^{1/3}(\log N)^3\Big)\] where $N$ is the product of all the distinct prime divisors of $xyz$. Furthermore, with the recent refinements of Yu in \cite{kr4} it is possible to solve completely the generalization of a problem of Erd\H{o}s to Lucas and Lehmer numbers; the original conjecture of Erd\H{o}s from 1965 states that $P(2^n-1)/n\rightarrow\infty$ as $n\rightarrow\infty$, where $P(m)$ denotes the greatest prime divisor of $m$ for integers $m>1$.

The generalizations to linear forms in $p$-adic elliptic logarithms were solved by R\'emond and Urfels (see \cite{gf}), and refined by Hirata-Kohno and Takada (see \cite{k}). For higher dimensions in the $p$-adic setting, the best results up to date are due to Bertrand and Flicker. They stated some results concerning simple abelian varieties or abelian varieties of CM-type (see \cite{db2} and \cite{f1}). Flicker also obtained a lower bound for linear forms on general abelian varieties but the bound is ineffective (see \cite{f2}).

The goal of this paper is to generalize the result on $p$-adic linear forms when evaluating at an algebraic point from a commutative algebraic group of positive dimension satisfying a technical condition and the condition of semistability. To describe the main theorem, let $K$ be a number field and $G$ a commutative algebraic group such that $G$ and the additive group $\mathbb G_a$ are disjoint over $K$ (see Section \ref{sav} for the definition of this notion). There are many commutative algebraic groups satisfying this property, for example the direct product of any finite copies of the multiplicative group $\mathbb G_m$ or any abelian variety. More generally we prove that every semi-abelian variety also satisfies the property.

Let $p$ be a prime number and consider embeddings $K\hookrightarrow \overline{\mathbb Q}\hookrightarrow \mathbb C_p$. Denote by $v$ the $p$-adic valuation which is the restriction of the $p$-adic valuation on $\mathbb C_p$ to $K$
and $K_v$ the completion of $K$ with respect to $v$. We embed $G$ into the projective space $\mathbb P_K^N$ for some positive integer $N$ and let $\Lie(G)$ denote the Lie algebra of $G$. Fixing a choice of basis for the vector space $\Lie(G)$ one can identify $\Lie(G)$ with the vector space $K^n$; here $n$ is the dimension of $G$. We get the normalized analytic function of the exponential map of $G(K_v)$ (with respect to the basis) consisting of $N$ functions analytic on a certain neighbourhood of $0$ in $K_v^n$. Let $W$ be the hyperplane in $K^n$ defined over $K$ by the linear form $$l(Z_1,\ldots,Z_n)=\beta_1Z_1+\cdots+\beta_nZ_n,$$ where $\beta_1,\ldots,\beta_n$ are elements, not all zero, in $K$. Let $u$ be an element in the above neighbourhood such that its image in the $p$-adic Lie group $G(K_v)$ is an algebraic point $\gamma$ in $G(K)$. The problem we consider is to give a lower bound for $|l(u)|_p$ when $l(u)$ is non-zero, here as usual we denote by $|\cdot|_p$ the $p$-adic absolute value on $\mathbb C_p$. The purpose of this paper is to solve the problem in the case when $(G,W)$ is semistable over $\overline{\mathbb Q}$. Here we use the condition of semistability introduced in \cite{aw} over the algebraic closure $\overline{\mathbb Q}$ since it concerns field extensions of the ground field $K$. Our lower bound consists of two parts; the first one consists of effectively computable constants depending only on the group $G$, the field $K$ and the choice of basis for the Lie algebra of $G$, and the second one is the product of the absolute logarithmic (Weil) height of the linear form $l$, of the algebraic point $\gamma$ and of the prime number $p$.

The method used in this paper to solve the problem can certainly be applied to get new results in transcendence theory. We leave this as a topic for a forthcoming paper.

In Section 2 we shall state the new result in detail. In Section 3 we collect some preliminary results including a Schwarz lemma in the $p$-adic domain, simple facts on disjointness and semistability, on heights, on the analytic representation of the exponential map and a fact about the order of vanishing of analytic functions. In Section 4 we shall give the proof of the main result of Section 2. The proof starts by embedding $G$ into some projective space; this involves a choice which we fix for the rest of the paper. We also choose a basis for the hyperplane. Then we work out the standard program in transcendence theory: we construct an auxiliary function with bounded height and with high order vanishing at certain points. Using the Schwarz lemma we can extrapolate and derive an upper bound. Liouville's inequality from Diophantine approximation gives a lower bound provided that we have non-vanishing. Algebraic considerations (namely multiplicity estimates) give the non-vanishing. Finally, comparing upper and lower bound gives the desired result by an appropriate choice of the parameters.

\section{New result}

As was mentioned above the $p$-adic theory of logarithmic forms has already been developed systematically with nice applications in number theory. It is therefore natural and clearly motivated to generalize the problem to the case of higher dimensions. There are several results in this direction due to R\'emond, Urfels, Hirata-Kohno, Takada, Flicker, Bertrand and others. However, the results only deal with elliptic curves or abelian varieties. We shall give here a new generalization to a class of commutative algebraic groups.

Let $K$ be a number field over $\mathbb{Q}$ and let $\mathcal O_K$ be the ring of algebraic integers of $K$. We choose an embedding $K\hookrightarrow \overline {\mathbb Q}$. Let $p$ be a prime number in $\mathbb Z$. We denote by $\mathbb Q_p$ the field of $p$-adic numbers and $\mathbb C_p$ the completion of the algebraic closure of ${\mathbb Q_p}$. We get the embedding $\sigma: K\hookrightarrow \mathbb C_p$ defined by the composition of the embeddings $K\hookrightarrow \overline {\mathbb Q}$ and $\overline{\mathbb Q}\hookrightarrow \mathbb C_p$. We therefore identify each element $x\in K$ with $\sigma(x)\in \mathbb C_p$. Let $v$ be the valuation on $K$ given by
$$v(x):=-\dfrac{\log |x|_p}{\log p},\quad \forall x\in K.$$
Denote by $K_v$ the completion of $K$ with respect to $v$. By completing the algebraic closure we get $K\hookrightarrow K_v\hookrightarrow \mathbb C_p$, which preserves the absolute values.
Let $G$ be a commutative algebraic group defined over $K$ of dimension $n\geq 1$. According to \cite{se}, see also \cite{fw} where explicit embeddings are constructed using exponential- and Theta-functions, $G$ can be embedded into some projective space $\mathbb P^N$. Let $L:\{1,\ldots,n\}\rightarrow \Lie(G)$ be a basis, $f_L=(f_1,\ldots,f_N)$ the normalized analytic function of the exponential map of $G(K_v)$ with respect to $L$ and $\Exp$ the map as defined in Section \ref{Exp}. We know that $f_1,\ldots,f_N$ are analytic on an open disk $\Lambda_v$ of $K_v^n$ (see again Section \ref{Exp}).
Let $W$ be the hyperplane in $K^n$ defined over $\mathcal O_K$ by the linear form
$$l(Z_1,\ldots,Z_n)=\beta_1Z_1+\cdots+\beta_nZ_n,$$
where $\beta_1,\ldots,\beta_n$ are elements, not all zero, in $\mathcal O_K$. Let $u$ be an element in $\Lambda_v$ such that $\ga:=\Exp(u)$ is an algebraic point in $G(K)$. Let $B$ and $H$ be fixed numbers such that
$$B\geq \max_{i=1,\ldots,n}\{3, H(\beta_i)\},\quad H\geq \max\{3, H(\ga)\}.$$
Put $b=\log B$ and $h=\log H$.
If $u=(u_1,\ldots,u_n)$ is not contained in $W_v:=W\otimes_KK_v$, i.e. $l(u)=\beta_1u_1+\cdots+\beta_nu_n\neq 0$, then a natural question is ``\emph{What can we say about lower bounds for $|l(u)|_p$?}".
Below we give an answer to this question in the case when $G, \mathbb G_a$ are disjoint over $K$ (for example, $G$ is semi-abelian, see Lemma \ref{sa}) and $(G,W)$ is semistable over $\overline{\mathbb Q}$. Let $\cd$ be the denominator of $L$ which is defined in Section \ref{Exp} and let $B^n(r_p|\cd|_p)$ denote the set $\{x=(x_1,\ldots,x_n)\in \mathbb C_p^n; |x_i|_p<r_p|\cd|_p$ for $i=1,\ldots,n\}$, where $r_p:=p^{-1/(p-1)}$. Then we have the following:
\begin{theorem}\label{mth}
Let $K$ be a number field and $G$ a commutative algebraic group of dimension $n\geq 1$ defined over $K$ such that $G$ and $\mathbb G_a$ are disjoint over $K$ and such that $(G,W)$ is semistable over $\overline{\mathbb Q}$.
There is a positive number $\omega_L$ depending on $L$ and there exist effectively computable positive real constants $c_0$ and $c_1$ independent of $b,h$ and $p$ with the following property:
\\
{\rm 1.} If $u\in \Lambda_v\cap B^n(r_p|\dd|_p)$ such that $\Exp(u)$ is an algebraic point in $G(K)$ then $l(u)=0$ or
$$\log|l(u)|_p>-c_0\oL^{n+3}bh^n(\log b+\log h)^{n+3}\log p.$$
{\rm 2.} If $u\in \Lambda_v$ such that $\Exp(u)$ is an algebraic point in $G(K)$ then we put
$$\n:=\max\Big\{0,\Big[\dfrac{1}{p-1}-v(u)\Big]+1\Big\}$$ and either $l(u)=0$ or we get the lower bound
$$\log |l(u)|_p>-c_1\oL^{n+3}bh^n(\log b+\log h+2\n\log p)^{n+3}\log p.$$
\end{theorem}
Throughout the paper constants do not depend on $b,h$ and $p$. We write $A\ll B$ (resp. $A\gg B$) if there is an effectively computable positive constant $c$ such that $A\leq cB$ (resp. $A\geq cB$).

We remark that although in the above theorem we only consider the case $\beta_1,\ldots,\beta_n\in\mathcal O_K$ the theorem is still true for $\beta_1,\ldots,\beta_n\in K$. To see this, let $\delta_i$ be the denominator of $\beta_i$ for $i=1,\ldots,n$ and $\delta$ the least common multiple of $\delta_1,\ldots,\delta_n$. Put $\beta'_i:=\delta \beta_i$ for $i=1,\ldots,n$ and $l'=\delta l$. Then $\beta'_1,\ldots,\beta'_n\in \mathcal O_K$ and $|l(u)|_p=|\delta^{-1}|_p|l'(u)|_p\geq |l'(u)|_p.$ Using Lemma \ref{s} we get $\log\delta\leq \log(\delta_1\cdots\delta_n)=\log\delta_1+\cdots+\log\delta_n\ll b,$ and this gives $h(\beta'_i)=h(\delta\beta_i)\ll b$ for all $i=1,\ldots,n.$ Hence the statement follows by applying Theorem \ref{mth} to the linear form $l'$ and from the inequality $\log|l(u)|_p\geq \log|l'(u)|_p$.

We also remark that it would be nice to remove the technical assumptions concerning disjointness and semistability in the statement. This clearly needs some further efforts. Since the paper is already quite long, we leave this for future work.

\section{Background and preliminaries}

In this section we discuss some basic background material which we need later for the proof of the main theorem.
\subsection{Some $p$-adic analysis}

The main result of this section is a Schwarz lemma in the $p$-adic domain which will be given in Proposition \ref{4} below. For any subfield $F$ of $\mathbb C_p$ and for any $R\geq 0$ we set $B_F(R):=\{x\in F; |x|_p<R\}$ and $\overline B_F(R):=\{x\in F; |x|_p\leq R\}$. From now on, we will skip the subscript $F$ when $F=\mathbb C_p$. Let $f(x)=\sum_na_nx^n$ be an analytic function on $\overline B(r)$ with $r>0$. We define
$$|f|_r:=\sup_{n}|a_n|_pr^n=\max_{n}|a_n|_pr^n.$$
We start with the remark that the function $z-a$ satisfies
$|z-a|_r=r$ for $r>0$ and for $a\in \overline B_F(r)$. Indeed, by definition we have $|z-a|_r=\max\{|a|_p, r\}=r.$
\begin{lemma}\label{w}
Let $f$ be an analytic function on $\overline B_F(r)$ with $r>0$, and $s,t$ real numbers such that $0<s\leq t\leq r$. If $f$ has $k$ zeros in the disk $\overline B_F(s)$ then
$$|f|_{s}\leq \left(\dfrac{s}{t}\right)^k|f|_{t}.$$
\end{lemma}
\begin{proof}
The statement is trivially true if $f\equiv 0$. Otherwise, the Weierstrass preparation theorem (see Theorem 2.14 in \cite{hp}) says that one can write $f=P\cdot g$ with
$P(z)=(z-a_1)\cdots(z-a_k)$
for $a_1,\ldots,a_k\in \overline B_F(s)$
and with a certain analytic function $g$ on $\overline B_F(r)$. By the remark above we get
$$|P|_{s}=|z-a_1|_{s}\cdots|z-a_k|_{s}=s^k$$
and similarly for $|P|_{t}$. Hence
$$|f|_{s}=s^k|g|_{s}\leq s^k|g|_{t}=\left(\dfrac{s}{t}\right)^k t^k|g|_{t}=\left(\dfrac{s}{t}\right)^k|f|_{t},$$
and this ends the proof.
\end{proof}
\begin{lemma}\label{w11}
Let $f$ be an analytic function on $\overline B(r)$ with $r>0$ and $s,t$ be real numbers such that $0<s\leq t\leq r$. Let $m$ be the number of zeros (counted with multiplicities) of $f$ in $B(t)$ then
$$|f|_{t}\leq \left(\dfrac{t}{s}\right)^m|f|_{s}.$$
\end{lemma}
\begin{proof}
The statement is trivial if $f\equiv 0$ or $s=t$. Otherwise, let $b_1,\ldots,b_m$ be the zeros of $f$ in $B(t)$ (counted with multiplicities) and let $t'$ be a fixed real number such that
$$\max\{|b_1|_p,\ldots,|b_m|_p\}<t'<t.$$
Let $l$ be the number of zeros (counted with multiplicities) of $f$ in $\overline B(s)$. Without loss of generality, we may assume that $b_1,\ldots,b_l$ are the $l$ zeros of $f$ in $\overline B(s)$.
By the Weierstrass preparation theorem (see Theorem 2.14 in \cite{hp}) there are $\alpha_1,\alpha_2\in \mathbb C_p$ and functions $g_1,g_2$ such that $g_1$ is analytic on $\overline B(s)$
and $g_2$ is analytic on $\overline B(t)$, $g_1(0)=g_2(0)=1$, $|g_1|_{s}=|g_2|_r=1$, and $f(z)=\alpha_1(z-b_1)\cdots(z-b_l)g_1=\alpha_2(z-b_1)\cdots (z-b_m)g_2$. Combining this with the above remark we get
$$|f|_{s}=|\alpha_1|_{s}|z-b_1|_{s}\cdots|z-b_l|_{s}|g_1|_{s}=|\alpha_1|_p s^l$$
and
$$|f|_{t'}=|\alpha_2|_{t'}|z-b_1|_{t'}\cdots|z-b_m|_{t'}|g_2|_{t'}=|\alpha_2|_pt'^m.$$
Hence
$$|f|_{t}=\lim_{t'\rightarrow t}|f|_{t'}=|\alpha_2|_p t^m.$$
On the other hand, since $g_1(0)=g_2(0)=1$ it follows that
$$f(0)=\alpha_1(-1)^lb_1\cdots b_l=\alpha_2(-1)^mb_1\cdots b_m$$
which leads to
$$|\alpha_1|_p=|\alpha_2|_p|b_{l+1}\cdots b_{m}|_p.$$
This shows that
$$\dfrac{|f|_{t}}{|f|_{s}}=\dfrac{|\alpha_2|_p}{|\alpha_1|_p}\dfrac{t^m}{s^l}=\dfrac{t^m}{s^m}\dfrac{s^{m-l}}{|b_{l+1}\cdots b_{m}|_p}.$$
Since $b_{l+1},\ldots,b_{m}\in B(t)\setminus \overline B(s)$ it follows that
$$|b_{l+1}\cdots b_{m}|_p\geq s^{m-l}.$$
Hence
$$\dfrac{|f|_{t}}{|f|_{s}}\leq \dfrac{t^m}{s^m},$$
and this is equivalent to the inequality
$$|f|_{t}\leq \left(\dfrac{t}{s}\right)^m|f|_{s}$$
which proves the statement.
\end{proof}
We are now able to prove the following proposition which is called {Schwarz lemma}.
\begin{proposition}\label{4}
Let $t\geq s$ be positive real numbers, $f$ an analytic function on $\overline B_F(t)$ and $\Gamma$ a finite subset  of $\overline B_F(s)$ of cardinality $l\geq 2$. We define
$$\delta:=\inf\{|\gamma-\gamma'|_p; \gamma, \gamma'\in \Gamma, \gamma\neq \gamma'\}$$
and
$$\mu:=\sup\{|f^{(n)}(\gamma)|_p; n=0,\ldots,k-1, \gamma\in\Gamma\}$$
with a positive integer $k$ and with $f^{(n)}$ the $n$-th derivative of $f$.
Assume that $|\delta|_p\leq 1$ then
$$|f|_{s}\leq \max\left\{\left(\dfrac{s}{t}\right)^{kl}|f|_{t}, \mu \left(\dfrac{s}{\delta}\right)^{kl-1}r_p^{-(k-1)}\right\}.$$
\end{proposition}
\begin{proof}
The proposition is trivially true if $f\equiv 0$ and therefore we may assume that $f$ is non-zero. If $f$ has at least $kl$ zeros in the disc $\overline B(s)$ then Lemma \ref{w} gives
$$|f|_{s}\leq \left(\dfrac{s}{t}\right)^{kl}|f|_{t}.$$
Otherwise $f$ has at most $kl-1$ zeros in the disc $\overline B(s)$. By the definition of $\delta$, the sets $B(\gamma,\delta),\gamma\in \Gamma$ are disjoint. In fact, suppose that there exist two distinct elements $\gamma_1$ and $\gamma_2$ in $\Gamma$ such that there is $x\in B(\gamma_1,\delta)\cap B(\gamma_2,\delta)$. Then this leads to the following contradiction
$$|\gamma_1-\gamma_2|_p\leq\max\{|x-\gamma_1|_p,|x-\gamma_2|_p\}<\delta.$$
Furthermore these $l$ sets $B(\gamma,\delta), \gamma\in\Gamma$, are subsets of $\overline B(s)$ since $\Gamma$ is contained in $\overline B_F(s)$, and this shows that there exists an element $\gamma_0$ of $\Gamma$ such that $f$ has at most $k-1$ zeros in $B(\gamma_0,\delta)$.
Since $\gamma_0\in \overline B_F(s)$ it gives
$|f(z-\gamma_0)|_{r}=|f(z)|_{r}$ for any $r$ such that $s\leq r\leq t$. We may therefore assume without loss of generality that $\gamma_0=0$.
Let $n(\delta,f)$ be the number of zeros (counted with multiplicities) of $f$ in $B(\delta)$. It is clear that $n(\delta,f)\leq k-1$, and this shows that
$$|f|_{\delta}=\sup_{n\leq k-1}\left|\dfrac{f^{(n)}(0)}{n!}\right|_p\delta^n.$$
On the other hand, it is known that
$$\left|\dfrac{1}{n!}\right|_p\leq p^{\frac{n-1}{p-1}}= r_p^{-(n-1)}\leq r_p^{-(k-1)}.$$
Combining this with $|\delta|_p\leq 1$, we get
$$|f|_{\delta}\leq \mu r_p^{-(k-1)}.$$
Finally, since $f$ has at most $kl-1$ zeros in $\overline B(s)$, Lemma \ref{w11} gives
$$|f|_{s}\leq \left(\dfrac{s}{\delta}\right)^{kl-1}|f|_{\delta}$$
and this shows the proposition.
\end{proof}

\subsection{Semi-abelian varieties}\label{sav}

Let $G$ be an algebraic group defined over a field $K$. It is well-known from Chevalley's theorem that there is a unique short exact sequence of algebraic groups
$$1\rightarrow H\rightarrow G\rightarrow A\rightarrow 1$$
with $H$ a linear algebraic group and $A$ an abelian variety defined over $K$. We call $G$ a {semi-abelian variety} if in the above exact sequence the group $H$ is a torus, i.e. $H_{\overline K}\cong (\mathbb G_m\otimes \overline{K})^k$ for some $k\geq 0$; here $\mathbb G_m$ denotes the multiplicative group. One can show that $G$ is semi-abelian defined over $K$ if and only if $G_{\overline K}$ is semi-abelian defined over $\overline{K}$. It is known that every semi-abelian variety is commutative (see \cite[Proposition 2.3]{fy}).
We recall the following definition which is given in a paper of W\"ustholz and Masser (see \cite{w10}): Let $G_1,\ldots,G_k$ be algebraic groups defined over $K$. We say that they are {(mutually) disjoint over $K$}
if every connected algebraic $K$-subgroup $H$ of $G:=G_1\times\cdots\times G_k$ has the form $H_1\times\cdots\times H_k$ for algebraic $K$-subgroups $H_1,\ldots,H_k$ of $G_1,\ldots,G_k$ respectively.
\begin{lemma}\label{s1}
For $S$ semi-abelian $Hom(S,\mathbb G_a)=(0).$
\end{lemma}
\begin{proof}
Notice that $S(\overline{K})_{tor}$ is Zariski dense in $S$ and that any homomorphism $\alpha$ maps $S(\overline K)_{tor}$ to $\mathbb G_a(\overline K)_{tor}=(0).$ It follows that $\alpha(S)=(0)$ and this gives $\alpha=0$.
\end{proof}
\begin{lemma}\label{sa}
Every semi-abelian variety defined over $K$ and the additive group $\mathbb G_a$ are disjoint over $K$.
\end{lemma}
\begin{proof}
Let $\mathscr H$ be an arbitrary algebraic $K$-subgroup of $\mathscr G:=\mathbb G_a\times G$. By making a base change to $\overline K$ we may assume, without loss of generality, that $K=\overline K$. We denote by $\pi_a$ and $\pi$ the projections of $\mathscr H$ on $\mathbb G_a$ and on $G$ respectively. Put $H_a:=\pi_a(\mathscr H\cap (\mathbb G_a\times\{e\}))$ and $H:=\pi(\mathscr H\cap (\{0\}\times G))$. Then $H_a$ is an algebraic $K$-subgroup of $\mathbb G_a$ and $H$ is an algebraic $K$-subgroup of $G$. Let $P$ be the image of $\mathscr H$ under the projection
$$\mathbb G_a\times G\rightarrow (\mathbb G_a\times G)/(H_a\times H)\cong (\mathbb G_a/H_a)\times (G/H).$$
Define $p_a$ and $p$ the projections of $(\mathbb G_a/H_a)\times (G/H)$ onto $\mathbb G_a/H_a$ and onto $G/H$ respectively. We show that $P\cong p_a(P)$ and $P\cong p(P)$. For the first isomorphism, since $p_a$ is surjective it is sufficient to show the restriction of $p_a$ to $P$ is injective. In fact, let $(x,y)$ be any element in $\mathscr H$ such that $p_a((x,y)(H_a\times H))=H_a$, and this means that $x\in H_a$. But $H_a=\pi_a(\mathscr H\cap (\mathbb G_a\times\{e\}))$, it follows that $(x,e)\in \mathscr H$. Combining this with $(x,y)\in\mathscr H$ we imply that $(0,y)\in\mathscr H$. Thus $y=\pi(0,y)\in \pi(\mathscr H\cap (\{0\}\times G))=H$, and this shows that $(x,y)\in H_a\times H$. By the same argument, we also get the second isomorphism.

Since $G$ is semi-abelian $G/H$ is semi-abelian as well. It follows from above that $P\cong p(P)$ is semi-abelian. By Lemma \ref{s1} we get $Hom(P,\mathbb G_a)=(0).$ Furthermore, it is clear that $H_a$ is either trivial or $\mathbb G_a$ hence $p_a(P)\subseteq \mathbb G_a$. This says that
$p_a\in Hom(P,\mathbb G_a)=(0)$ which gives $P\cong p_a(P)=(0)$ and implies that $\mathscr H=H_a\times H$.
\end{proof}

\subsection{Semistability}

We recall the following notion which is due to W\"ustholz (see \cite[Chapter 6]{aw}).
Let $G$ be an algebraic group defined over a field $K$ and $V$ a $K$-linear subspace of the Lie algebra $\Lie(G)$ of $G$. We associate with $(G,V)$ the index
\begin{equation*}
\tau(G,V):=
\begin{cases} \dfrac{\dim V}{\dim G}\quad {\textnormal {\rm if }} \dim G>0,\\
1 \hspace*{1.4cm}{\textnormal{\rm otherwise}}.
\end{cases}
\end{equation*}
The pair $(G,V)$ is called {semistable (over $K$)} if for any proper quotient $\pi: G\rightarrow H$ defined over $K$, we have $\tau(G,V)\leq \tau(H,\pi_*(V))$ where $\pi_*:\Lie(G)\rightarrow \Lie(H)$ is the $K$-linear map induced by $\pi$. Let $F/K$ be a field extension. We say that $(G,V)$ is {semistable over $F$} if $(G_F,V\otimes_KF)$ is semistable.

\subsection{Heights}

Let $K$ be a number field of degree $d$ over $\mathbb Q$, and $M_K$ the set of places of $K$. For a place $v\in M_K$ we write $K_v$ for the completion of $K$ at $v$ and introduce the normalized absolute value $\vert\cdot\vert_v$ as follows. If $v\mid p$ we define $|p|_v:=p^{-[K_v:\mathbb Q_p]}$. If $v\mid\infty$ it corresponds to the embedding $\tau_v$ of $K$ into $\mathbb C$, and we define $|x|_v:=|\tau_v(x)|^{[K_v:\mathbb R]}$ for any $x\in K_v$. One can show that
$$\prod_{v\in M_K}|x|_v=1,\quad\forall x\in K\setminus\{0\},$$
and this is called the {product formula}. Let $P\in \mathbb P^n(K)$ be a point represented by a homogeneous non-zero vector $x$ with coordinates $x_0,\ldots,x_n$. We set
$$h_K(x):=\sum_{v\in M_K}\max_{i}\log |x_i|_v.$$
The absolute {logarithmic (Weil) height} $H$ on $\mathbb P^n(\overline{\mathbb Q})$ is defined by
$$h(P):=\dfrac{1}{[K:\mathbb Q]}h_K(x)$$
where $K$ is any number field containing $P$, and the {absolute (Weil) height} of $P$ is defined by $H(P):=e^{h(P)}$.

Let $\alpha$ be an element in $\overline{\mathbb Q}$. We define $h(\alpha)$ as the absolute logarithmic height of the point in $\mathbb{P}^1(K)$ with projective coordinates $1,\alpha$. It is known that $h(\alpha_1\cdots\alpha_r)\leq h(\alpha_1)+\cdots+h(\alpha_r)$ and $h(\alpha_1+\cdots+\alpha_r)\leq \log r+h(\alpha_1)+\cdots+h(\alpha_r)$ with $r\geq 1$ and with $\alpha_1,\ldots,\alpha_r\in\overline{\mathbb Q}$. Let $x=(x_1,\ldots,x_n)$ be an element in $\mathbb A^n(K)$. We define
$$|x|_v:=\max_i|x_i|_v, \quad\forall v\in M_K,$$
and
$$\hm(x):=\sum_{v\in M_K}\log|x|_v$$
for $x\neq 0$, otherwise we put $\hm(0):=0$.
It is convenient to introduce the function
$$h_{L^2}(x):=\sum_{v\in M_K}\log |x|_{L^2,v}$$
where
\begin{equation*}
|x|_{L^2,v}=
\begin{cases}
\max_i|x_i|_v  \hspace*{1.5cm} v \text{ non-archimedean }\\
\big(\sum_i\tau_v(x_i)^2\big)^{\frac{1}{2}} \hspace*{0.71cm} v \text{ real }\\
\sum_i\tau_v(x_i)\overline{\tau_v(x_i)}\hspace*{0.53cm} v \text { complex. }
\end{cases}
\end{equation*}
We write $\log^{+}t$ for $\max\{0,\log t\}$ for any positive real number $t$, extended by $\log^{+}0=0$. Put
$$H_{\max}^+:=\prod_{v\in M_K}\max\{|x|_v,1\},$$
$$\hp(x):=\log H_{\max}^+(x)=\sum_{v\in M_K}\log^+|x|_v,$$
and
$$h_{L^2}^+(x)=\sum_{v\in M_K}\log^+|x|_{L^2,v}.$$
These heights are related by
$$\hm\leq h_{L^2}\leq \hm+\dfrac{d}{2}\log(n+1)$$
and
$$\hp\leq h^+_{L^2}\leq \hp+\dfrac{d}{2}\log(n+1).$$
If we identify each point $x=(x_1,\ldots,x_n)\in \mathbb A^n(K)$ with the projective point $(1:x_1:\ldots:x_n)$ then by definition one gets $h_K(x)=\hp(x).$

One can extend the notations given above to polynomials in $n$ variables $T_1,\ldots,T_n$ with coefficients in $K$. In more details, let $P=\sum_ia_iT^i$ be such a polynomial with $i: \{1,\ldots,n\}\rightarrow \mathbb N^n$ a multi-index and $T^i=T_1^{i(1)}\cdots T_n^{i(n)}$. It corresponds to a point $a=(\ldots,a_i,\ldots)$ in an affine space $\mathbb A^N(K)$ and we define
$$|P|_v:=|a|_v,\quad |P|_{L^2,v}:=|a|_{L^2,v}$$
and the heights of $P$ as $0$ for $P=0$ and for $P\neq 0$ as
$$\hm (P)=\sum_{v\in M_K}\log |P|_v,\quad h_{L^2}(P)=\sum_{v\in M_K}\log |P|_{L^2,v}.$$
We shall also use
$$\hp(P)=\sum_{v\in M_K}\log^+|P|_v,\quad h^+_{L^2}(P)=\sum_{v\in M_K}\log^+|P|_{L^2,v}.$$
\begin{proposition}[Siegel's lemma]
Let $N>M$ be positive integers and let $l_1,\ldots,l_M$ be linear forms in $N$ variables in $T_1,\ldots,T_N$ with coefficients in $K$. Then there exists a non-trivial solution $x=(x_1,\ldots,x_N)\in \mathcal O_K^N$ for the system of linear equations $l_1(T_1,\ldots,T_N)=\cdots=l_M(T_1,\ldots,T_N)=0$ such that
$$\hp(x)\leq \dfrac{1}{2}\log|\dis (K)|+M/(N-M)\max_{i} h_{L^2}(l_i)$$
where $\dis (K)$ denotes the field discriminant of $K$.
\end{proposition}
\begin{proof}
This is Corollary 11 of \cite{bv}.
\end{proof}
We recall the Liouville's inequality for number fields which is simple but has an important role in the proof of the main theorem below.
\begin{proposition}[Liouville's inequality]
Let $K$ be a number field and let $\alpha$ be a non-zero element in $K$. Then
$$\log|\alpha|_v\geq -\dfrac{h(\alpha)}{[K:\mathbb Q]},\quad \forall v\in M_K.$$
\end{proposition}
\begin{proof}
This is \cite[Corollary 2.9.2]{bom}.
\end{proof}
For an algebraic number $\alpha\in K$, the {denominator} $\delta$ of $\alpha$ is defined as the smallest positive integer for which the element $\delta\alpha$ is in $\mathcal O_K$. For a polynomial $P$ with coefficients $a_i, i\in I$, in $K$, we define the denominator $\delta(P)$ of $P$ as the smallest positive integer for which the elements $\delta(P)a_i\in \mathcal O_K$ for all $i\in I$. The following lemma gives an inequality between the height and the denominator of an algebraic number.
\begin{lemma}\label{s}
Let $\alpha$ be an element in $K$ and $\delta$ its denominator. One has
$$\log\delta\leq \dfrac{h(\alpha)}{[K:\mathbb Q]}.$$
\end{lemma}
\begin{proof}
For $v\in M_K\setminus M_K^{\infty}$ let $p$ be the residue characteristic of $v$. By definition
$$|\alpha|_v=|N_{K_v/\mathbb Q_p}(\alpha)|_p^{\frac{1}{[K_v:\mathbb Q_p]}}=|N_{K_v/\mathbb Q_p}(\alpha)|_p^{\frac{1}{n_v}}$$
with $n_v$ the degree of $K_v$ over $\mathbb Q_p$. Since $N_{K_v/\mathbb Q_p}(\alpha)$ is an element in $\mathbb Q_p$ and since the value group of $\mathbb Q_p$ is $\mathbb Z$, the element
$$m_v:=\dfrac{n_v}{\log p}\max\{\log |\alpha|_v,0\}$$ is a non-negative integer.
Let $S$ be the set $\{(p,v); p$ the residue characteristic of $v, v\in M_K\setminus M_K^{\infty}, |\alpha|_v >1\}$. One has $S$ is a finite set. We see that
$$\prod_{(p,v)\in S}p^{m_v}\alpha\in \mathcal O_K.$$
This shows, by definition of the denominator of $\alpha$, that
$$\delta\leq \prod_{(p,v)\in S}p^{m_v}$$
and therefore
$$\log\delta\leq \dfrac{h(\alpha)}{[K:\mathbb Q]}.$$
The lemma is proved.
\end{proof}

\subsection{Analytic representation of exponential maps}\label{Exp}

Let $K$ be a number field and let $G$ be an algebraic group defined over $K$. We denote by $\overline{G}$ the Zariski closure of $G$ in $\mathbb P^N$. Let $U$ be the open affine subset defined by $\overline{G}\cap \{X_0\neq 0\}$. We know that the affine algebra $\Gamma(U,\mathcal O_{\overline G})$ is stable under the action of any element in $\frak g=\Lie(G)$ and it is generated by $\xi_1,\ldots,\xi_N$, where
$$\xi_i:=\left(\dfrac{X_i}{X_0}\right)\!\!\bigg\vert_U,\quad \forall i=1,\ldots,N$$
(see \cite{w1}). We call a map $L:\{1,\ldots,n\}\rightarrow \frak g$ a {basis} if $L(1),\ldots,L(n)$ is a basis for $\frak g$. With such a basis $L$, one gets a system of polynomials $P_{i,L(j)}$ in $N$ variables such that
$$L(j)\xi_i=P_{i,L(j)}(\xi_1,\ldots,\xi_N),\quad \forall i=1,\ldots,N, \forall j=1,\ldots,n.$$
This means that
$$\mathcal L_j:=L(j)(\mathcal O_K[\xi_1,\ldots,\xi_N])$$
is an $\mathcal O_K$-module in $K[\xi_1,\ldots,\xi_N]$ for any $j=1,\ldots,n$. Put $\mathcal L=\mathcal L_1+\cdots+\mathcal L_n$ and define
$$\mathcal I_L:=(\mathcal O_K[\xi_1,\ldots,\xi_N]:\mathcal L)=\{t\in \mathcal O_K; t\mathcal L\subset \mathcal O_K[\xi_1,\ldots,\xi_N]\}.$$
Then $\mathcal I_L$ is an ideal of $\mathcal O_K$ and its norm $N_{K:\mathbb Q}(\mathcal I_L)$ is an ideal in $\mathbb Z$ which has to be principal. It takes the form $(\cd)$ for some positive integer $\cd$. We call $\cd$ the {denominator} of $L$.

Denote by $\p_1,\ldots,\p_n$ the canonical basis of $\Lie(K_v^n)$ defined as $\p_ix_j=\delta_{ij}$ for all $i=1,\ldots,n$ and for all $j=1,\ldots,N$, where $\delta_{ij}$ are Kronecker's delta and $x_i$ are the coordinate functions of $K_v^n$. We define the isomorphisms
 $$\p: K_v^n\rightarrow \Lie(K_v^n),\quad x=(x_1,\ldots,x_n)\mapsto x_1\p_1+\cdots+x_n\p_n$$
and
$$\iota: \Lie(K_v^n)\rightarrow \Lie(G(K_v)),\quad \iota(\p_1)=\Dl(1),\ldots,\iota(\p_n)=\Dl(n).$$
We consider now the set $G(K_v)$ of $K_v$-points of $G$. It is known that $G(K_v)$ is a Lie group over $K_v$. By \cite[Chapter III, \S 7]{Bourbaki}, there is a map $\exp$ (which is called exponential map) defined and locally analytic on an open disk $U_v$ of $\Lie (G(K_v))$. The functions
$$f_i:=\xi_i\circ \Exp,\quad i=1,\ldots,N$$
are analytic on $\Lambda_v:=(\iota\circ \p)^{-1}(U_v)$ in $K_v^n$, where $\Exp=\exp\circ \iota\circ \p$.
\\
Let $\mathcal O_{G(K_v)}, \mathcal O_{U_v}, \mathcal O_{\p(\Lambda_v)}$ and $\mathcal O_{\Lambda_v}$ be the sheaves of analytic functions on $G(K_v), U_v, \p(\Lambda_v)$ and $\Lambda_v$, respectively. So we get commutative diagrams
$$\xymatrix{\mathcal O_{G(K_v)}\ar[d]^{L(j)}\ar[r]^{\exp^*}&\mathcal O_{U_v}\ar[r]^{\iota^*}&\mathcal O_{\p(\Lambda_v)}\ar[r]^{\p^*} &\mathcal O_{\Lambda_v}\ar[d]^{\p_j}
\\\mathcal O_{G(K_v)}\ar[r]^{\exp^*}&\mathcal O_{U_v}\ar[r]^{\iota^*}&\mathcal O_{\p(\Lambda_v)}\ar[r]^{\p^*}&\mathcal O_{\Lambda_v}}$$
for all $j=1,\ldots,n$. This leads to
$$\big(\p_j\circ \Exp^* \big)(\xi_i)=\big(\Exp^*\circ L(j)\big)(\xi_i),\quad \forall i=1,\ldots, N,$$
i.e.
$$\p_j(f_i)=L(j)(\xi_i)\circ\Exp=P_{i,L(j)}(\xi_1,\ldots,\xi_N)\circ\Exp=P_{i,L(j)}(f_1,\ldots,f_N)$$
for any $i=1,\ldots,N$ and $j=1,\ldots,n$.

The map $f_L=(f_1,\ldots,f_N): \Lambda_v\rightarrow K_v^N$ is called the {normalized analytic representation} of the exponential map $\exp$ with respect to the basis $L$. We define
$$\ccc:=\max_{i,j}\deg P_{i,L(j)};\quad e_L:=v(\delta_L);\quad \cc:=\max_{i,j}h(P_{i,L(j)})$$
and
$$\oL:=\max\{1,e_L\}(\cc+\log\cd+\log\ccc);$$
here by convention, $\log \ccc =0$ if $\ccc =0$.

We fix the following notations. For $m=(m_1,\ldots,m_k)\in\mathbb{N}^k$ with $0\leq k\leq n$, we write
$$\partial^m:=\partial_1^{m_1}\cdots \partial_k^{m_k};\quad L^m:=L(1)^{m_1}\cdots L(k)^{m_k};\quad |m|:=m_1+\cdots+m_k.$$
\begin{lemma}\label{ww}
Let $L:\{1,\ldots,n\}\rightarrow \frak g$ be a basis and $P(T_1,\ldots,T_N)$ a polynomial in $N$ variables with coefficients in $K$ of total degree $\leq D$. Let $T$ be a non-negative integer and $t=(t_1,\ldots,t_n)\in\mathbb{N}^n$ be such that $T=t_1+\cdots+t_{n}$. There exists a polynomial
$P_t\in K[T_1,\ldots,T_N]$ such that
$$\partial^tP(f_1,\ldots,f_N)=P_t(f_1,\ldots,f_N),$$
satisfying\\
\hspace*{0.8cm} {\rm 1.} $\deg P_t\leq D+T(\ccc-1),$ \\ \hspace*{0.8cm} {\rm 2.} $\log|P_t|_v\ll \log|P|_v+T(\cc+\log (D+T\ccc)), \quad\forall v\in M_K.$
\end{lemma}
\begin{proof}
We shall prove the lemma by induction on $T=|t|$. The lemma is trivially true for $|t|=0$.
Assume that it is true for any $t\in\mathbb{N}^n$ with $|t|=T\geq 0$. We prove it is also true for any $t\in\mathbb{N}^n$ with $|t|=T+1$.
Let $t=(t_1,\ldots,t_n)\in\mathbb{N}^n$ be such that $t_1+\cdots+t_n=T+1$. Without loss of generality, we may assume that $t_1\geq 1$. Put $\tau=(t_1-1,\ldots,t_n)$, by induction one gets
$$\partial^\tau P(f_1,\ldots,f_N)=P_{\tau}(f_1,\ldots,f_N)$$
with
$$D_\tau:=\deg P_\tau\leq D+T\ccc$$ and
$$\log|P_\tau|_v\ll \log|P|_v+T(\cc+\log (D+T\ccc)).$$
We write
$$P_{\tau}=\sum_{m_1+\cdots+m_N\leq D_\tau}a(m_1,\ldots,m_N)T_1^{m_1}\cdots T_N^{m_N}=\sum_ma(m)T_1^{m_1}\cdots T_n^{m_n}$$
and
$$P_{i,L(1)}=\sum_{m_{i,1}+\cdots+m_{i,N}\leq\ccc}a(m_{i,1},\ldots,m_{i,N})T_1^{m_{i,1}}\cdots T_N^{m_{i,N}}$$
with the coefficients $a(m_{i,1},\ldots,m_{i,N})\in K$ for all $1\leq i\leq N$.
This gives
$$\partial_1f_i=\sum_{m_{i,1}+\cdots+m_{i,N}\leq\ccc}a(m_{i,1},\ldots,m_{i,N})f_1^{m_{i,1}}\cdots f_N^{m_{i,N}},\quad \forall i=1,\ldots,N.$$
Since $\partial^t=\partial_1\partial_1^{t_1-1}\cdots\partial_n^{t_n}=\partial_1\partial^\tau$ it follows that
\begin{equation*}\begin{split}
\partial^tP(f_1,\ldots,f_N)&=\partial_1\partial^\tau P(f_1,\ldots,f_N)=\partial_1P_\tau(f_1,\ldots,f_N)\\
&=\sum_ma(m)\sum_{i=1}^Nm_i\Big(\prod_{j\neq i}f_j^{m_j}\Big)f_i^{m_i-1}\partial_1f_i
\end{split}\end{equation*}
which is expanded as
$$\sum_m\sum_{i=1}^N\sum_{m_{i,1}+\cdots+m_{i,N}\leq \ccc}m_ia(m)a(m_{i,1},\ldots,m_{i,N})\Big(\prod_{j\neq i}f_j^{m_j+m_{i,j}}\Big)f_i^{m_i+m_{i,i}-1}.$$
This shows that
$$\partial^tP(f_1,\ldots,f_N)=P_t(f_1,\ldots,f_N)$$
for a certain polynomial
$$P_t(T_1,\ldots,T_N)=\sum_{l}q(l)T_1^{l_1}\cdots T_N^{l_N}$$
with $q(l)=\sum m_ia(m)a(m_{i,1},\ldots m_{i,N})$; here the sum is taken over the set
$\{(m_1,\ldots,m_N,i,m_{i,1},\ldots,m_{i,N}); m_j+m_{i,j}=l_j$ for $j\neq i$ and $m_i+m_{i,i}=l_i+1, 1\leq i\leq N, m_{i,1}+\cdots+m_{i,N}\leq \ccc, m_1+\cdots+m_N\leq D_\tau\}$
such that
\begin{equation*}\begin{split}
\deg P_t&\leq \max_i(m_1+\cdots+m_N+m_{i,1}+\cdots+m_{i,N}-1)\\
&\leq D_\tau+\ccc-1\leq D+T(\ccc-1)+\ccc-1\\
&\leq D+(T+1)(\ccc-1).
\end{split}\end{equation*}
Furthermore we find that
\begin{equation*}\begin{split}
|q(l)|_v&\leq \sum m_i|a(m)|_v|a(m_{i,1},\ldots,m_{i,N})|_v\\
&\leq (\ccc+1)^ND_\tau |P_\tau|_v\max_{i,j} |P_{i,L(j)}|_v.
\end{split}\end{equation*}
This shows that
\begin{equation*}\begin{split}
\log|q(l)|_v&\leq N\log(\ccc+1)+\log D_\tau+\log|P_\tau|_v+\cc\\
&\ll \log|P|_v+T(\cc+\log (D+T\log\ccc))+N\log(\ccc+1)+\cc\\
&\ll \log|P|_v+(T+1)(\cc+\log (D+(T+1)\ccc))
\end{split}\end{equation*}
for all $v\in M_K$, and the lemma follows.
\end{proof}
Let $k$ be a non-negative integer. We define $\mathcal L(k)$ as the sum of images of $\mathcal O_K[\xi_1,\ldots,\xi_N]$ under all differentials of order $\leq k$, i.e.
$$\mathcal L(k):=\sum_{t\in\mathbb Z^n_{\geq 0}; |t|\leq k}L^t(\mathcal O_K[\xi_1,\ldots,\xi_N]).$$
Let $\mathcal I(k)$ be the ideal $(\mathcal O_K[\xi_1,\ldots,\xi_N]:\mathcal L(k))$ in $\mathcal O_K$. We get the following lemma.
\begin{lemma}\label{int}
$$\mathcal I(k)\supset (\mathcal I_L)^k,\quad \forall k\in\mathbb{N}.$$
\end{lemma}
\begin{proof}
We shall prove this by induction on $k$. If $k=0$, the lemma is trivially true. Assume it is also true for $k=m\geq 0$. One has to show that
$$a_1\cdots a_{m+1}L^{t}(\xi_i)\in \mathcal O_K[\xi_1,\ldots,\xi_N]$$
for $i=1,\ldots,n$, for $a_1,\ldots,a_{m+1}\in \mathcal I_L$ and for $t=(t_1,\ldots,t_n)\in\mathbb{N}^n$ with $|t|=m+1$. There is at least one $j\in\{1,\ldots, n\}$ such that $t_j\geq 1$. Put $\tau=(t_1,\ldots,t_{j-1},t_j-1,t_{j+1},\ldots,t_n)$. We see that
$$a_1\cdots a_{m+1}L^{t}(\xi_i)=a_1\cdots a_{m}L^{\tau}(a_{m+1}L(j)(\xi_i)).$$
Since $a_{m+1}\in \mathcal I_L$ it follows that
$$a_{m+1}L(j)(\xi_i)=Q_{i,j}(\xi_1,\ldots,\xi_N),\quad \forall i=1,\ldots,N$$
for some polynomials $Q_{i,j}(T_1,\ldots,T_N)$ with coefficients in $\mathcal O_K$.
By induction with $|\tau|=m$, we have $a_1\cdots a_{m}\in \mathcal I_L^{m}\subset \mathcal I(m)$. In particular,
$$a_1\cdots a_{m}L^{\tau}(Q_{i,j}(\xi_1,\ldots,\xi_N))\in \mathcal O_K[\xi_1,\ldots,\xi_N].$$
The lemma is therefore proved.
\end{proof}
\begin{lemma}\label{lem03}
For $t=(t_1,\ldots,t_n)\in\mathbb{N}^n$ with $|t|=T$ and for a polynomial $P(T_1,\ldots,T_N)\in \mathcal O_K[T_1,\ldots,T_N]$ we have
$$\delta_L^T\p^tP(f_1,\ldots,f_N)\in \mathcal O_K[f_1,\ldots,f_N].$$
Hence $\delta_L^T\p^t f_i(0)\in \mathcal O_K$ for $i=1,\ldots,N$.
\end{lemma}
\begin{proof}
There exists a polynomial $P_t(T_1,\ldots,T_N)$ with coefficients in $K$ such that
$$L^tP(\xi_1,\ldots,\xi_N)=P_t(\xi_1,\ldots,\xi_N).$$
By Lemma \ref{int}, we see that the polynomial $\delta^T_LP_t$ has coefficients in $\mathcal O_K$.
Note that
$$\p^tP(f_1,\ldots,f_N)=P_t(f_1,\ldots,f_N),$$
and then one gets
$$\delta_L^T\p^tP(f_1,\ldots,f_N)\in \mathcal O_K[f_1,\ldots,f_N].$$
Finally, since $f_i(0)=0$ for $i=1,\ldots, N$ it follows that
$$\delta_L^T\p^t f_i(0)=P_t(f_1(0),\ldots,f_N(0))\in \mathcal O_K,\quad \forall i=1,\ldots,N.$$
\end{proof}
\begin{proposition}\label{pro1}
The functions $f_i$ satisfy
$$|f_i(x)|_p<1,\quad \forall x\in B^n(|\dd|_pr_p).$$
\end{proposition}
\begin{proof}
It follows from the previous lemma and by considering the Taylor expansion of $f_i$ at $0$ together with the fact $|n!|_p\geq r_p^{n-1}$ for all positive integers $n$.
\end{proof}

\subsection{The order of vanishing of analytic functions}

In this section let $F$ denote a complete subfield of $\mathbb C_p$. Let $V$ be a vector subspace of $\Lie(G(F))$ and $f$ a non-zero $p$-adic analytic function on a neighborhood of a point $z\in F^n$. We say that {$f$ has a zero at $z$ of order $\geq T$ along $V$} if $(v_1\cdots v_kf)(z) =0$ for any $0\leq k<T$ and for any $v_1,\ldots,v_k\in V$. We also say that {$f$ has a zero at $z$ of exact order $T$ along $V$} if it has order $\geq T$ at $z$ along $V$ and furthermore, there are $w_1,\ldots,w_T$ in $V$ such that $({w_1}\cdots{w_T}f)(z)\neq 0.$
\begin{proposition}
With notations as above, let $d$ be the dimension of $V$ and let $\Delta_1,\ldots,\Delta_d$ be a basis for $V$. Then $f$ has a zero at $z$ of order $\geq T$ along $V$ if and only if $(\Delta_1^{t_1}\ldots\Delta_d^{t_d}f)(z) =0$ for $(t_1,\ldots,t_d)\in\mathbb N^d$ with $t_1+\cdots+t_d<T$ and $f$ has a zero at $z$ of exact order $T$ if it has order $\geq T$ at $z$ along $V$ and furthermore, there is a $d$-tuple $\tau=(\tau_1,\ldots,\tau_d)\in\mathbb{N}^d$ such that $|\tau|=T$ and $(\Delta_1^{\tau_1}\ldots\Delta_d^{\tau_d}f)(z)\neq 0$.
\end{proposition}
\begin{proof}
We prove the first statement. In fact, it suffices to show that if $(\Delta_1^{t_1}\cdots\Delta_d^{t_d}f)(z) =0$ for any $(t_1,\ldots,t_d)\in\mathbb N^d$ with $t_1+\cdots+t_d<T$, then $f$ has a zero at $z$ of order $\geq T$ along $V$. Let $k$ be integer such that $0\leq k<T$ and $v_1,\ldots,v_k$ arbitrary elements of $V$. For $i=1,\ldots,k$ one can write $v_i=a_{i1}\Delta_1+\cdots+a_{id}\Delta_d$ with $a_{i1},\ldots,a_{id}\in F$. For $t=(t_1,\ldots,t_d)\in\mathbb{N}^d$ with $|t|<T$, we expand
$$(v_1\cdots v_kf)(z)=\Big(\prod_{i=1}^k(a_{i1}\Delta_1+\cdots+a_{id}\Delta_d)f\Big)(z)=\sum_{\alpha\in I}a_{\alpha}(\Delta_1^{\alpha_1}\cdots \Delta_d^{\alpha_d}f)(z).$$
Since $k<T$ it follows that $|\alpha|=\alpha_1+\cdots+\alpha_d<T$ for every $\alpha\in I$. Hence the sum vanishes, and this shows the first statement. It is clear that the second statement follows at once from the definition and the first statement.
\end{proof}

\section{Proofs}

\subsection {Proof of the second statement of Theorem \ref{mth}}

We shall show that the first assertion of the theorem implies the second one. Let $u\in \Lambda_v$ such that $\Exp(u)$ is an algebraic point in $G(K)$. We define
$$\n:=\max\Big\{0,\Big[\dfrac{1}{p-1}-v(u)\Big]+1\Big\},$$
and $u':=p^{\n}u$. Then $u'\in \Lambda_v$ and
$$|u'|_p=|p^{\n}|_p|u|_p=p^{-\n-v(u)}=p^{\frac{1}{p-1}-v(u)-\n}r_p<r_p.$$
Moreover, if $l(u)\neq 0$ then $l(u')=p^{\n}l(u)\neq 0$ and applying the first statement of Theorem \ref{mth} to $u'$ in $\Lambda_v\cap B^n(r_p|\delta_L|_p)$ one gets
$$\log|l(u')|_p>-c_0\oL^{n+3}bh'^n(\log b+\log h')^{n+3}\log p;$$
here $h':=\max\{1,h(\gamma')\}$ with $\gamma':=\Exp(u')=p^{\n}\Exp(u)=\gamma^{p^{\n}}$ where $\gamma:=\Exp(u)$. By \cite [Prop. 5]{se} one has
$$h\big(\gamma^{p^{\n}}\big)\leq (p^{\n })^2h(\gamma)\leq p^{2\n }h$$
and this implies that $h'\leq p^{2\n}h.$
Hence
$$\n\log p+\log|l(u)|_p>-c_0\oL^{n+3}bh^n(\log b+\log h+2\n\log p)^{n+3}\log p.$$
We therefore conclude that
$$\log |l(u)|_p>-c_1\oL^{n+3}bh^n(\log b+\log h+2\n\log p)^{n+3}\log p$$
for some positive constant $c_1$.

\subsection{A projective embedding}

Following \cite{se} (cf. also \cite{fw} and \cite{waa}), there exist a positive integer $N$ and an embedding $\varphi: G\hookrightarrow \mathbb P^N$ of the group from above, which is defined over a number field $K$ of degree $m$. Without loss of generality, we may assume that the identity element $e\in G(K)$ under $\varphi$ has coordinates $(1:0:\ldots:0)$ in $\mathbb P^N$.
\begin{lemma}
There exists an embedding $\psi: G\rightarrow \mathbb P^N$ defined over a number field of degree $m(N+1)$ such that $\psi(e)=(1:0:\ldots:0)$ and $X_0(\psi(g))\neq 0$ for all $g\in G(K)$, where $X_0$ denotes the first projective coordinate on $\mathbb P^N$.
\end{lemma}
\begin{proof}
We choose a field extension $K_1$ of $K$ of degree $N+1$ and a basis $\epsilon_0,\ldots,\epsilon_N$ of $K_1$ over $K$. The degree of the extension $K_1\supseteq\mathbb Q$ is therefore $m(N+1)$. It is clear that the vectors
$$(\epsilon_0,0,\ldots,0), (-\epsilon_1,\epsilon_0,0,\ldots,0),\ldots,(-\epsilon_N,0,\ldots,0,\epsilon_0)$$
form a basis of $K_1^{N+1}$ which gives rise to a unique element in GL$_{N+1}(K_1)$ mapping this basis to the standard basis of $K_1^{N+1}$. This linear isomorphism is expressed explicitly by the matrix
\begin{equation*}
A=
\begin{pmatrix} \epsilon_0^{-1}&\epsilon_0^{-2}\epsilon_1&\ldots &\epsilon_0^{-2}\epsilon_N
\\
0&\epsilon_0^{-1}&\ldots &0
\\
\vdots &\vdots &\ddots &\vdots
\\
0 & 0 &\ldots &\epsilon_0^{-1}
\end{pmatrix}.
\end{equation*}
We let $\psi$ be the composition of $A$ with the embedding $\varphi$ above. Then $\psi(e)$ has projective coordinates $(1:0:\ldots:0)$ and $X_0(\psi(g))\neq 0$ for all $g\in G(K)$. Indeed, let $(x_0:x_1:\ldots:x_N)$ be a projective coordinate of $\varphi(g)$. By the construction of $\psi$, we obtain
\begin{equation*}
\begin{split}
\psi(g)&=(\epsilon_0^{-1}x_0+\epsilon_0^{-2}\epsilon_1x_1+\cdots+\epsilon_0^{-2}\epsilon_Nx_N:\epsilon_0^{-1}x_1:\ldots:\epsilon_0^{-1}x_N)
\\
&=(\epsilon_0x_0+\epsilon_1x_1+\cdots+\epsilon_Nx_N:\epsilon_0x_1:\ldots:\epsilon_0x_N).
\end{split}
\end{equation*}
Thus we see that $\psi(e)=(1:0:\ldots:0)$. In addition, since $\epsilon_0,\ldots,\epsilon_N$ is a basis of $K_1$ over $K$ and $x_0,\ldots,x_N$ are in $K$, not all zero, it follows that $X_0(\psi(g))$ is non-zero.
Note that the embedding $\psi$ is defined over $K_1$.
\end{proof}
We shall fix the embedding $\psi: G\hookrightarrow \mathbb P^N$ for the rest of the paper and identify each element $g\in G$ with its image $\psi(g)$ in $\mathbb P^N$.
By \cite[Section 2]{w1}, there is a finite field extension $K_2$ of $K_1$ (the degree of this extension is a positive constant) with the following property: There exist bihomogeneous polynomials $E_0,\ldots,E_N$ in $Z_0,\ldots,Z_N$ and $X_0,\ldots,X_N$ of bidegree $(b,b)$ with coefficients in $K_2$ and their height bounded from above by a positive constant, and a Zariski open set $U\subset G\times G$ containing $\Gamma(\gamma)\times \Gamma(\gamma)$ such that for $(g,g')\in U$ the homogeneous coordinates of $g+g'$ are $(E_0(g,g'):\ldots:E_N(g,g'))$; here $\Gamma(\gamma)$ denotes the subgroup generated by $\gamma$ in $G(K)$ with $\gamma:=\Exp(u)$. The degree of the extension $K_2$ over $K$ is also a positive constant. We may therefore assume, without loss of generality, that $K$ is already equal to $K_2$ and has degree $d$ over $\mathbb Q$. We call $(E_1,\ldots,E_N)$ an {addition formula} for $G$ and from now on we fix such an addition formula $E=(E_1,\ldots,E_N)$.

\subsection{Basis of the hyperplane}

We define the linear form in $n+1$ variables
$$\LL(Z_0,Z_1,\ldots,Z_n):=Z_0-l(Z_1,\ldots,Z_n).$$
This gives the vector space
$$\W:=\{(z_0,z_1,\ldots,z_n)\in K_v^{n+1}; z_0=l(z_1,\ldots,z_n)\}\subset K_v^{n+1}.$$
Let $e_1,\ldots,e_n$ be the basis for $\W$ defined by
$$e_1=(\beta_1,1,0,\ldots,0),e_2=(\beta_2,0,1,0,\ldots,0),\ldots, e_n=(\beta_n,0,\ldots,0,1).$$
This gives differential operators (corresponding to the isomorphism $\partial$ introduced in Section \ref{Exp})
$$\diff_{1}=\partial(e_1)=\beta_1\partial_0+\partial_1,\diff_{2}=\partial(e_2)=\beta_2\partial_0+\partial_2, \ldots, \diff_{n}=\partial(e_n)=\beta_n\partial_0+\partial_n;$$
here $\partial_0,\ldots,\partial_n$ is the standard basis for $\Lie(K_v^{n+1})$.
Let $\uu:=(0,u_1,\ldots,u_n)$ and $\w:=(u_0,u_1,\ldots,u_n)$ be vectors in $K_v^{n+1}$ with $u_0:=l(u)$. Then
$$\w=u_1e_1+\cdots+u_ne_n$$
and this shows that $\w\in \W$. We furthermore see that
$$\w-\uu=(l(u),0,\ldots,0).$$
Define
$$\diff^t:=\Delta^{t_1}_{1}\cdots \Delta^{t_n}_{n}$$
for $t=(t_1,\ldots,t_n)\in \mathbb{N}^n.$

\subsection{The auxiliary function}

In this section we shall construct an auxiliary polynomial by using Siegel's lemma.
Let $\G:=\mathbb G_a\times G$ be the product of the additive group $\mathbb G_a$ with $G$. The exponential map of the Lie group $\G(K_v)$ is
$\exp_{\G(K_v)}=\id_{K_v}\times \exp.$ Note that for $u\in \Lambda_v$ we have $X_0(\Exp(u))\neq 0$; here the map $\Exp:\Lambda_v\rightarrow G(K_v)$ is defined in Section \ref{Exp}. We introduce the function
$$\Psi_P:=(\id_{K_v}\times \Exp)^*P\Big(Y,1,\dfrac {X_1}{X_0},\ldots,\dfrac {X_N}{X_0}\Big)$$
for each polynomial $P$ in $N+2$ variables $Y,X_0,\ldots,X_N$. This means that $\Psi_P(w)=P(y,1,f_1(x_1,\ldots,x_n),\ldots,f_N(x_1,\ldots,x_n))$ is analytic on $K_v\times \Lambda_v^n$, where $w=(y,x)\in K_v^{n+1}$ with $x=(x_1,\ldots,x_n)\in\Lambda_v^n.$

We define the {order $\ord_{g,\W}P$ of $P$ at $g=(\id_{K_v}\times \Exp)(w)$ along $\W$} to be infinity if $\Psi_P$ is identically zero in a neighborhood of $x$, and to be the order of $\Psi_P$ at $w$ along $\mathscr W$, otherwise.

Let $S_0,D_0,D,T$ be positive integers. We apply Siegel's lemma to construct a polynomial $P$ in $N+2$ variables with coefficients in $\mathcal O_K$ such that $P$ does not vanish identically on $\G$ and has height $h(P)$ bounded from above by a quantity in terms of $L,S_0,D_0,D,T,b,h$. We further require that $\ord_{s\uu,\W}\Psi_{P}\gg T$ for all $0\leq s<S_0$.
\begin{proposition}\label{siegel1}
There are positive constants $c_2$ and $c_3$ such that if $D_0D^{n}\geq c_2S_0T^{n}$ there is a polynomial $P$ in $N+2$ variables $Y,X_0\ldots,X_{N}$ with coefficients in $\mathcal O_K$, homogeneous in $X_0,\ldots,X_{N}$ of degree $D$, and with $\deg P_{Y}\leq D_0$ such that
\\
\hspace*{0.8cm}{\rm 1.} $P$ does not vanish identically on $\mathscr G$,
\\
\hspace*{0.8cm}{\rm 2.} $(\diff^{t}\Psi_P)(s\uu)=0, 0\leq s<S_0, t=(t_1,\ldots,t_n), 0\leq t_1,\ldots,t_n<2T$,
\\
\hspace*{0.8cm}{\rm 3.} $h(P)\leq c_3(T(\cc+\log\cd+\log(D+T\log\ccc))+D_0b+DS_0^2h).$
\end{proposition}
\begin{proof}
Since the dimension of $G$ is $n$, without loss of generality, we may assume that the homogeneous coordinates $X_0,\ldots,X_{n}$ are algebraically independent modulo the ideal of $G$. We shall construct a non-zero polynomial $P$ in $n+2$ variables $Y$ and $X_0,\ldots, X_n$ which is homogeneous in $X_0,\ldots, X_n$ of degree $D$ (this polynomial therefore satisfies 1. in the proposition) such that $\deg _{Y}P\leq D_0$ and such that 2. and 3. in the proposition are satisfied. Such a polynomial can be written in the form
$$P(Y,X)=\sum_{i=0}^{D_0}\sum_{j=1}^{D_1}p_{ij}Y^iM_j(X_0,\ldots,X_n),$$
where $D_1$ is the number of homogeneous monomials of degree $D$ in the $n+1$ variables $X_0,\ldots,X_n$ and $M_1,\ldots,M_{D_1}$ run through all these monomials. An easy computation shows that $D_1=\binom{D+n}n$. For short, we write $\Psi$ for $\Psi_P$.
Let $E=(E_1,\ldots,E_N)$ be the addition formula for $G$ from above. By abuse of notation, we put
$$E_i(z,x):=E_i(1,f_1(z),\ldots,f_N(z), 1, f_1(x),\ldots,f_N(x)),$$
for $z,x$ in $\Lambda_v$. For $y\in K_v$ we also define
$$\Psi_{s}(y,x):=\Psi(y,su+x)E_0(su,x)^D.$$
Put
$$I:=\{(s,t); 0\leq s<S_0, t=(t_1,\ldots,t_n), 0\leq t_1,\ldots,t_n<2T\}.$$
For any $(s,t)\in I$
we shall determine the coefficients $p_{ij}$ such that
$$(\diff^{t}\Psi_{s})(0,0)=0,\quad \forall (s,t)\in I.$$
By the property of the addition formula $E$, for any $x$ in a neighbourhood of $0$ small enough so that $E(su,x)\neq 0$, one gets
$$f_i(su+x)=\dfrac{E_i(su,x)}{E_0(su,x)},\quad i=1,\ldots,N.$$
This leads to
\begin{equation*}\begin{split}M_j(1,f_1(su+x),\ldots,f_n(su&+x))
\\
&=M_j\left(1,\dfrac{E_1(su,x)}{E_0(su,x)},\ldots,\dfrac{E_n(su,x)}{E_0(su,x)}\right)\\&=E_0(su,x)^{-D}M_j(E_0(su,x),\ldots,E_n(su,x)).
\end{split}\end{equation*}
Therefore one gets
$$\Psi_{s}(y,x)=\Psi(y,su+x)E_0(su,x)^D=\sum_{i,j}p_{ij}y^iM_j(E_0(su,x),\ldots,E_n(su,x)).$$
On the other hand, for each $s$, we can express $E_i(su,x)$ as
$$E_i(su,x)=F_i(f_1(x),\ldots,f_N(x)),\quad i=0,\ldots,n,$$
here $F_i$ are polynomials in $N$ variables with polynomials (which have coefficients in $K$) in the $f_1(su),\ldots,f_N(su)$ as coefficients. Since
$$\gamma^s=\Exp(su)=(1:f_1(su):\ldots:f_N(su))$$
and since $h(\gamma^s)\ll s^2h$ (see \cite [Prop. 5]{se}) we may estimate the height of polynomials $h(F_i)\ll s^2h$ for $i=0,\ldots,n$. One can therefore choose a common denominator $\den_s\ll s^2h$ for the polynomials $F_0,\ldots,F_n$. Since $M_j$ is a monomial of degree $D$, there is a polynomial $Q_{j,s}$ in $N$ variables of degree $\ll D$ with $\log|Q_{j,s}|_v \ll Ds^2h$ for $v\in M_K$ such that
$$M_j(E_0(su,x),\ldots,E_n(su,x))=Q_{j,s}(f_1(x),\ldots,f_N(x))$$
for each $j=1,\ldots,D_1$. Then
$$\Psi_{s}(y,x)=\sum_{i,j}p_{ij}y^iQ_{j,s}(f_1(x),\ldots,f_N(x)),$$
this gives
$$(\Delta^t\Psi_{s})(0,0)=\sum_{i,j}p_{ij}\big(\Delta^t(y^iQ_{j,s}(f_1,\ldots,f_N))\big)(0,0).$$
Define
$$a_{ij}^{st}:=\big(\Delta^t(y^iQ_{j,s}(f_1,\ldots,f_N))\big)(0,0)$$
for $i=0,\ldots,D_0$, $j=1,\ldots,D_1$ and $(s,t)\in I$. Note that $\partial_0=\partial/\partial y$. We expand
\begin{equation*}\begin{split}a^{st}_{i,j}&=\big(\diff_1^{t_1}\cdots \diff_n^{t_n}(y^iQ_{j,s}(f_1,\ldots,f_N))\big)(0,0)\\
&=\big((\beta_1\partial_0+\partial_1)^{t_1}\cdots (\beta_n\partial_0+\partial_n)^{t_n}(y^iQ_{j,s}(f_1,\ldots,f_N))\big)(0,0)\\
&=\sum_{i_1=0}^{t_1}\cdots\sum_{i_n=0}^{t_n}\binom{t_1}{i_1}\cdots\binom{t_n}{i_n}\beta_1^{t_1-i_1}\cdots\beta_n^{t_n-i_n}\\
&\hspace*{1.0cm}\cdot\Big(\Big(\dfrac{\partial}{\partial y}\Big)^{(t_1+\cdots+t_n)-(i_1+\cdots+i_n)}\partial_1^{i_1}\cdots\partial_n^{i_n}(y^iQ_{j,s}(f_1,\ldots,f_N))\Big)(0,0)\\
&=\sum_{i_1=0}^{t_1}\cdots\sum_{i_n=0}^{t_n}\binom{t_1}{i_1}\cdots\binom{t_n}{i_n}\beta_1^{t_1-i_1}\cdots\beta_n^{t_n-i_n}\\
&\hspace*{1.0cm}\cdot\Big(\Big(\dfrac{\partial}{\partial y}\Big)^{(t_1+\cdots+t_n)-(i_1+\cdots+i_n)}y^i\Big)(0)\Big(\partial_1^{i_1}\cdots\partial_n^{i_n}(Q_{j,s}(f_1,\ldots,f_N))\Big)(0).
\end{split}\end{equation*}
For $m\in \mathbb{N}^n$ we obtain from Lemma \ref{ww}
$$\p^m (Q_{j,s}(f_1,\ldots,f_N))=Q_{j,s,m}(f_1,\ldots,f_N)$$
for some polynomial $Q_{j,s,m}$ in $N$ variables with
$$\log|Q_{j,s,m}|_v\ll \log|Q_{j,s}|_v+|m|(\cc+\log(D+|m|\ccc))$$
$$\hspace*{3.85cm} \ll |m|(\cc+\log(D+|m|\ccc))+Ds^2h, \quad\forall v\in M_K.$$
This means that
$$\log\Big|\Big(\p^{m}(Q_{j,s}(f_1,\ldots,f_n)\Big)(0)\Big|_v\ll |m|(\cc+\log(D+|m|\ccc))+Ds^2h$$
for $v\in M_K$.
In particular,
\begin{equation*}\begin{split}
\log\Big|\Big(\partial_1^{i_1}\cdots\partial_n^{i_n}(Q_{j,s}&(f_1,\ldots,f_N))\Big)(0)\Big|_v
\\
&\ll T(\cc+\log(D+T\ccc))+Ds^2h,\quad \forall v\in M_K.
\end{split}\end{equation*}
Furthermore one gets
\begin{equation*}
\Big(\Big(\dfrac{\partial}{\partial y}\Big)^{(t_1+\cdots+t_n)-(i_1+\cdots+i_n)}y^i\Big)(0)=
\begin{cases}
0\hspace*{0.15cm} \text{ if } (t_1+\cdots+t_n)-(i_1+\cdots+i_n)\neq i,
\\
i! \hspace*{0.15cm} \text{ if } (t_1+\cdots+t_n)-(i_1+\cdots+i_n)=i.
\end{cases}
\end{equation*}
In other words, we have
$$\log\Big|\Big(\Big(\dfrac{\partial}{\partial y}\Big)^{(t_1+\cdots+t_n)-(i_1+\cdots+i_n)}y^i\Big)(0)\Big|_v\ll \log(T!)\ll T\log T,\quad\forall v\in M_K^\infty.$$
We deduce that
$$\log|a^{st}_{ij}|_v\ll  T(\cc+\log(D+T\ccc))+Ds^2h,\quad\forall v\in M_K^\infty.$$
Since $h(\beta_i)\leq b$ for $i=1,\ldots,n$, $\log|\beta_i|_v\leq b$ for $v\in M_K$. By noting that $\den_s\cd^{|m|}Q_{j,s,m}$ has coefficients in $\mathcal O_K$, we get $\den_s\cd^{2nT}a_{ij}^{st}$ is also in $\mathcal O_K$ and the quantity $\log|\den_s\cd^{2nT}a_{ij}^{st}|_v$ is
$$\ll D_0b+T(\log\cd+\cc+\log(D+T\log\ccc))+DS_0^2h$$
for $(s,t)\in I$ and for $v\in M_K^\infty$. We now consider the linear forms in $n_0:=D_0D_1$ variables $T_{ij}$
$$l_{st}:=\sum_{i,j}b^{st}_{ij}T_{ij},$$
where $b^{st}_{ij}:=\den_s\cd^{2nT}a_{ij}^{st}$ for all $(s,t)\in I$.
Let $m_0$ be the number of these linear forms, then $m_0\ll S_0T^n$ and $n_0=D_0D_1=D_0\binom{D+n}n\gg D_0D^n$. Since $b^{st}_{ij}\in \mathcal O_K$ we get
\begin{equation*}\begin{split}
\hm (l_{st})&=\sum_{v\in M_K^\infty}\log\max_{i,j}|b^{st}_{ij}|_v\\
&\ll D_0b+T(\cc+\log\cd+\log (D+T\log\ccc))+DS_0^2h.
\end{split}\end{equation*}
We now apply Siegel's lemma. It follows that under the condition $D_0D^n\gg S_0T^n$
there is a non-zero vector $p_0=(p_{ij})$ with coordinates in $\mathcal O_K$ such that $l_{st}(p_0)=0$ and
$$h(p_0)\leq \dfrac{m_0}{n_0-m_0}\max_{s,t}h_{L^2}(l_{st})$$
But using
$$h_{L^2}(l_{st})\ll \hm(l_{st})+\log n_0,$$
this gives
$$h(P)\ll D_0b+T(\cc+\log\cd+\log(D+T\ccc))+DS_0^2h.$$
It remains to show that $(\diff^{t}\Psi)(s\uu)=0$. In fact, since $l_{st}(p_0)=0$ one gets $(\diff^{t}\Psi_{s})(0,0)=0$ for $(s,t)\in I$.
Put
$$\Psi^*_{s}(y,x):=\Psi(y,su+x), \quad E_{s}(x):=E_0(su,x)^D,$$
then $\Psi^*_{s}=\Psi_{s}E_{s}^{-D}$. We therefore get by Leibnitz' rule for derivation that
\begin{equation*}\begin{split}(\diff^{t}\Psi)(s\uu)&=(\diff^{t}\Psi)(0,su)=(\diff^{t}\Psi^*_{s})(0,0)=\big(\diff^{t}(\Psi_{s}E_{s}^{-D})\big)(0,0)=0.\end{split}\end{equation*}
This completes the proof.
\end{proof}
From now on until Section \ref{choice}, we shall fix a polynomial $P$ as in Proposition \ref{siegel1} and let $\Psi=\Psi_P$ be the analytic function associated with $P$.

\subsection{Extrapolation}

In this section we use the $p$-adic Schwarz lemma to give an upper bound for $|(\diff^t\Psi)(s\w)|_p$ (with $|t|<T$). We need the following lemma.
\begin{lemma}\label{e}
Let $Q$ be a polynomial in $k+1$ variables $X_0,\ldots,X_k$ with coefficients in the ring $\Ov$ of algebraic integers of $K_v$ and $\deg_{X_0}Q\leq l$ with $l\in\mathbb N, l\geq 1$. Then
$$|Q(x_0,x)-Q(0,x)|_p\leq \max_{1\leq i\leq l}|x_0|_p^i,$$
for any $x_0\in K_v$ and $x\in \Ov^k$.
\end{lemma}
\begin{proof}
We define the polynomial $Q_x(X):=Q(X,x)$ in one variable $X$. By assumption and by the ultrametric inequality, we see that $Q_x$ has coefficients in $\Ov$. We write $Q_x(X)=a_lX^l+\cdots+a_0$, with $a_0,\ldots,a_l\in \Ov$. Then
\begin{equation*}\begin{split}
|Q_x(x_0)-Q_x(0)|_p&=|a_lx_0^l+\cdots+a_1x_0|_p\leq \max_{1\leq i\leq l}|a_ix_0^i|_p\leq \max_{1\leq i\leq l}|x_0|_p^i,
\end{split}\end{equation*}
and the lemma follows.
\end{proof}
\begin{lemma}\label{dis}
For $0\leq s<S$ and for $t=(t_1,\ldots,t_n)\in\mathbb{N}^n$ such that $0\leq t_1,\ldots,t_n<2T$ we have
$$|(\diff^t\Psi)(s\w)-(\diff^t\Psi)(s\uu)|_p\leq |\dd^{-1}|_{p}^{2nT}|l(u)|_p.$$
\end{lemma}
\begin{proof}
In fact, we can write again
$$P(Y,X_0,\ldots,X_N)=\sum_{i,j}p_{ij}Y^iM_j(X_0,\ldots,X_N).$$
Put $R_j(x)=M_{j}(1,f_1(x),\ldots,f_N(x))$, then
\begin{equation*}\begin{split}
\diff^t\Psi&=\sum_{i,j}p_{ij}\big(\diff^t(y^i R_j)\big)\\
&=\sum_{i,j}p_{ij}\big((\beta_1\partial_0+\partial_1)^{t_1}\cdots (\beta_n\partial_0+\partial_n)^{t_n}(y^i R_j)\big)\\
&=\sum_{i,j}p_{ij}\sum_{i_1=0}^{t_1}\cdots\sum_{i_n=0}^{t_n}\binom{t_1}{i_1}\cdots\binom{t_n}{i_n}\beta_1^{t_1-i_1}\cdots\beta_n^{t_n-i_n}\\
&\hspace*{3cm}\cdot\Big(\Big(\dfrac{\partial}{\partial y}\Big)^{(t_1+\cdots+t_n)-(i_1+\cdots+i_n)}\partial_1^{i_1}\cdots\partial_n^{i_n}(y^i R_j)\Big)\\
&=\sum_{i,j}p_{ij}\sum_{i_1=0}^{t_1}\cdots\sum_{i_n=0}^{t_n}\binom{t_1}{i_1}\cdots\binom{t_n}{i_n}\beta_1^{t_1-i_1}\cdots\beta_n^{t_n-i_n}\\
&\hspace*{3cm}\cdot\Big(\Big(\dfrac{\partial}{\partial y}\Big)^{(t_1+\cdots+t_n)-(i_1+\cdots+i_n)}y^i\Big)(\partial_1^{i_1}\cdots\partial_n^{i_n}R_j).
\end{split}\end{equation*}
Using the fact that
$$\Big(\dfrac{\partial}{\partial y}\Big)^{(t_1+\cdots+t_n)-(i_1+\cdots+i_n)}y^i=0 \quad
{\rm  if}\quad (t_1+\cdots+t_n)-(i_1+\cdots+i_n)>i,$$
and that $\partial_1^{i_1}\cdots\partial_n^{i_n}R_{j}$
is a polynomial in $f_1,\ldots,f_N$ with its denominator bounded from above by $\dd^{|t|}$ by Lemma \ref{lem03},
we deduce that
$$\dd^{|t|}\beta_1^{t_1-i_1}\cdots\beta_n^{t_n-i_n}\Big(\Big(\dfrac{\partial}{\partial y}\Big)^{(t_1+\cdots+t_n)-(i_1+\cdots+i_n)}y^i\Big)(\partial_1^{i_1}\cdots\partial_n^{i_n} R_j)$$
is a polynomial in $y,f_1,\ldots,f_N$ with coefficients in $\OK$.
On the other hand, the coefficients $p_{ij}$ are in $\OK$, and this implies that
$${\dd}^{|t|}(\diff^t\Psi)=Q_t(y,f_1,\ldots,f_N),$$
for some polynomial $Q_t(Y,X_1,\ldots,X_N)$ with coefficients in $\OK$, and with $\deg_YQ_t\leq D_0.$
This means that $|(\diff^t\Psi)(s\w)-(\diff^t\Psi)(s\uu)|_p$ is equal to
$$|\dd|_p^{-|t|}|Q_t(su_0,f_1(su),\ldots,f_N(su))-Q_t(0, f_1(su),\ldots,f_N(su))|_p.$$
Since $u\in \Lambda_v\cap B^n(r_p|\dd|_p)$, we get $|f_1(su)|_p,\ldots,|f_N(su)|_p<1$ by Proposition \ref{pro1} and, taking into account that $r_p|\dd|_p<1$ and that $\beta_1,\ldots,\beta_n\in \mathcal O_K$, we find that
$$|su_0|_p=|s|_p|u_0|_p\leq |u_0|_p=|\beta_1u_1+\cdots+\beta_nu_n|_p< 1.$$
By Lemma \ref{e} we obtain
\begin{equation*}\begin{split}
|(\diff^t\Psi)(s\w)-(\diff^t\Psi)(s\uu)|_p&\leq |\dd|_p^{-2nT}\max_{1\leq i\leq D_0}|su_0|_p^i\\
&\leq |\dd^{-1}|_p^{2nT}|u_0|_p=|\dd^{-1}|_p^{2nT}|l(u)|_p.
\end{split}\end{equation*}
This gives the assertion.
\end{proof}
\begin{proposition}\label{3}
For $0\leq s<S_0$ and for $t=(t_1,\ldots,t_n)\in\mathbb{N}^n$ such that $0\leq t_1,\ldots,t_n<2T$ the estimate
$$|(\diff^t\Psi)(s\w)|_p\leq |\dd^{-1}|_p^{2nT}|l(u)|_p$$
holds.
\end{proposition}
\begin{proof}
By Proposition \ref{dis}, for $0\leq s<S_0$ and for $t=(t_1,\ldots,t_n)\in\mathbb{N}^n$ with $0\leq t_1,\ldots,t_n<2T$ one has
$$|(\diff^t\Psi)(s\w)-(\diff^t\Psi)(s\uu)|_p\leq |\dd^{-1}|_p^{2nT}|l(u)|_p.$$
Moreover by Proposition \ref{siegel1}, for $0\leq s<S_0$ and for $t=(t_1,\ldots,t_n)\in\mathbb{N}^n$ with $0\leq t_1,\ldots,t_n<2T$, we have
$$(\diff^t\Psi)(s\uu)=0.$$
This gives
$$|(\diff^t\Psi)(s\w)|_p\leq |\dd^{-1}|_p^{2nT}|l(u)|_p$$
as claimed.
\end{proof}
For each $n$-tuple $t\in\mathbb{N}^n$ such that $|t|<T$ we introduce the function
$$f(z):=(\diff^{t}\Psi)(z\w)$$
in the variable $z$. It is analytic on $\overline B(1)$. Our next step is to apply Proposition \ref{4} to the function $f$. We shall prove an upper bound for the derivatives of $f$ on a certain finite set. Thanks to Proposition \ref{3}, one gets
\begin{proposition}\label{5}
For $\tau,s\in\mathbb Z$ such that $0\leq \tau<T$ and $0\leq s<S_0$ one has
$$|f^{(\tau)}(s)|_p\leq |\dd^{-1}|_p^{2nT}|l(u)|_p.$$
\end{proposition}
\begin{proof}
By recalling that $\w=u_1e_1+\cdots+u_ne_n$ and using the composition rule for derivatives we get
\begin{equation*}\begin{split}
f^{(\tau)}(z)&=\big(\left(u_0\partial_0+\cdots+u_n\partial_n\right)^\tau \diff^{t}\Psi\big)(z\w)\\
&=\Big(\big((\beta_1u_1+\cdots+\beta_nu_n)\partial_0+u_1\partial_1+\cdots+u_n\partial_n\big)^\tau\diff^t\Psi\Big)(z\w)\\
&=\Big(\big(u_1(\beta_1\partial_0+\partial_1)+\cdots+u_n(\beta_n\partial_0+\partial_n)\big)^\tau\diff^t\Psi\Big)(z\w)\\
&=\big((u_1\diff_1+\cdots+u_n\diff_n)^\tau\diff^{t}\Psi\big)(z\w).
\end{split}\end{equation*}
Since $|u_i|_p<1$ for $i=1,\ldots,n$, the multinomial expansion together with the ultrametric inequality gives
$$|f^{(\tau)}(z)|_p\leq \max_{0\leq i_1,\ldots,i_n\leq \tau; i_1+\cdots+i_n=\tau}|(\diff_1^{i_1}\cdots \diff_n^{i_n}\diff^{t}\Psi)(z\w)|_p.$$
Since $\tau$ and $|t|$ are $<T$ the assertion follows from Proposition \ref{3}.
\end{proof}
\begin{lemma}\label{value}
Let $\alpha$ be a non-zero element in $K_v$ such that $|\alpha|_p<p^{-1/(p-1)}$. Then
$$v(\alpha)-\dfrac{1}{p-1}\geq\dfrac{1}{2d^2}.$$
\end{lemma}
\begin{proof}
We know that $v(K_v^\times)=\dfrac{1}{d_v}\mathbb Z,$
with $d_v:=[K_v:\mathbb Q_p]$. Since $|\alpha|_p=p^{-v(\alpha)}<p^{-\frac{1}{p-1}}$
there is a positive integer $a$ such that
$$v(\alpha)=\dfrac{a}{d_v}>\dfrac{1}{p-1}.$$
This implies that $a(p-1)-d_v\geq 1$. If $p-1\geq 2d_v$ then
$$v(\alpha)-\dfrac{1}{p-1}\geq\dfrac{1}{d_v}-\frac{1}{p-1}\geq \dfrac{1}{2d_v}\geq \dfrac{1}{2d}\geq\dfrac{1}{2d^2}.$$
Otherwise, if $p-1<2d_v$ then
$$v(\alpha)-\dfrac{1}{p-1}=\dfrac{a(p-1)-d_v}{d_v(p-1)}\geq \dfrac{1}{d_v(p-1)}>\dfrac{1}{2d_v^2}\geq\dfrac{1}{2d^2}.$$
\end{proof}
From now on we put $\epsilon:=\frac{1}{3d^2}$.
Combining Lemma \ref{value} and Proposition \ref{5} together with Proposition \ref{4}, we get the following result.
\begin{proposition}\label{sch}
For $s\in\mathbb{N}$ and for $t=(t_1,\ldots,t_n)\in\mathbb{N}^n$ such that $|t|<T$ the quantity $|(\diff^t\Psi)(s\w)|_p$ is bounded from above by
$$p^{-(\epsilon S_0-e_L)T}\max\Big\{1, p^{\left((2n-1)e_L+\epsilon S_0+\frac{1}{p-1}\right)T}S_0^{S_0T}|l(u)|_p\Big\}.$$
\end{proposition}
\begin{proof}
As above, we consider the function $f(z)=(\diff^{t}\Psi)(z\w)$
in the variable $z$, and apply the $p$-adic Schwarz lemma to the function $f$. We first show that the function $f$ is analytic on $\overline B(R)$, where $R:=p^\epsilon$. It suffices to show that $zu_i\in B(r_p|\dd|_p)$ for $z\in \overline B(R)$ and for $i=1,\ldots,n$. In fact, if $u_i=0$ then it is trivially true. Otherwise, since $|\dd^{-1}u_i|_p<p^{-\frac{1}{p-1}}$
it follows from Lemma \ref{value} that
$$v(\dd^{-1}u_i)-\dfrac{1}{p-1}\geq \dfrac{1}{2d^2}.$$
Hence
$$v(\dd^{-1}u_i)-\epsilon=\dfrac{1}{p-1}+\Big(v(\dd^{-1}u_i)-\dfrac{1}{p-1}-\dfrac{1}{3d^2}\Big)>\dfrac{1}{p-1}$$
which leads to
$$R|\dd^{-1}|_p|u_i|_p=p^{\epsilon}p^{-v(\dd^{-1}u_i)}=p^{-(v(\dd^{-1}u_i)-\epsilon)}<p^{-\frac{1}{p-1}}$$
or equivalently to $R|u_i|_p<r_p|\dd|_p$. This means that $zu_i\in B(r_p|\dd|_p)$ for $z\in \overline B(R)$.
Next we establish an upper bound for $|f|_R$. As in the proof of Proposition \ref{dis} there is a polynomial $Q(Y,X_1,\ldots,X_N)$ with coefficients in $\OK$ such that $\deg_YQ\leq D_0$ and
$$f(z)=\dd^{-T}Q(zu_0,f_1(zu),\ldots,f_N(zu)).$$
We note that
$$|zu_0|_p=|\beta_1zu_1+\cdots+\beta_nzu_n|_p\leq |zu_1+\cdots+zu_n|_p\leq\max\{|zu_1|_p,\ldots,|zu_n|_p\}<1$$
and deduce from Proposition \ref{pro1} that $|f_i(zu)|_p<1$ for $i=1,\ldots,N$ and for $z\in \overline B(R)$.
This gives $|Q(zu_0,f_1(zu),\ldots,f_N(zu))|_p\leq 1$ which leads to
$$|f(z)|_p\leq |\dd^{-1}|_p^T,\quad \forall z\in \overline B(R).$$
In other words we have
$$|f|_R\leq |\dd^{-1}|_p^T.$$
Finally let $\Gamma$ be the set $\{s\in\mathbb Z; 0\leq s<S_0\}$ and $\delta$ be the minimum of $|s-s'|_p$ for $s\neq s'$ in $\Gamma$. The cardinality of $\Gamma$ is $S_0$ and we have $\delta\leq 1$.
We define
$$\mu:=\sup\{|f^{(\tau)}(s)|_p; 0\leq \tau<T, s\in \Gamma\}.$$
Using Lemma \ref{5} one gets $\mu\leq |\dd^{-1}|_p^{2nT}|l(u)|_p$.
We apply Proposition \ref{4} to the function $f$ to obtain
\begin{equation*}\begin{split}
|f|_1&\leq \max\Big\{\Big(\dfrac{1}{R}\Big)^{S_0T}|f|_{R}, \mu\Big(\dfrac{1}{\delta}\Big)^{S_0T-1}r_p^{-(T-1)}\Big\}\\
&\leq \max\big\{p^{-\epsilon S_0T}|\dd^{-1}|_p^T, |\dd^{-1}|_p^{2nT}|l(u)|_p\delta^{-(S_0T-1)}r_p^{-T}\big\}\\
&\leq \max\big\{p^{-\epsilon S_0T}p^{e_LT}, p^{2ne_LT}p^{\frac{T}{p-1}}\delta^{-(S_0T-1)}|l(u)|_p\big\}\\
&\leq \max\big\{p^{-(\epsilon S_0-e_L)T}, p^{\left(2ne_L+\frac{1}{p-1}\right)T}\delta^{-(S_0T-1)}|l(u)|_p\big\}.
\end{split}\end{equation*}
Moreover, for $s,s'\in \Gamma$ such that $s\neq s'$ one has
$$|s-s'|_p\geq\dfrac{1}{|s-s'|} >\dfrac{1}{S_0}.$$
This gives $\delta^{-1}<S_0$. Thus we obtain
\begin{equation*}\begin{split}
|f|_1&\leq \max\Big\{p^{-(\epsilon S_0-e_L)T}, p^{\left(2ne_L+\frac{1}{p-1}\right)T}S_0^{S_0T}|l(u)|_p\Big\}\\
&=p^{-(\epsilon S_0-e_L)T}\max\Big\{1, p^{\left((2n-1)e_L+\epsilon S_0+\frac{1}{p-1}\right)T}S_0^{S_0T}|l(u)|_p\Big\}.
\end{split}\end{equation*}
The proposition therefore follows from the fact that
$$|(\diff^t\Psi)(s\w)|_p=|f(s)|_p\leq |f|_1$$
for all integers $s\geq 0$.
\end{proof}
\begin{proposition}\label{upperbound}
There is a positive constant $c_4$ such that if
$$\log |l(u)|_p\leq -c_4\Big(\Big(S_0+\dfrac{1}{p-1}+e_L\Big)T\log p+S_0T\log S_0\Big)$$
then
$$\log |(\diff^{t}\Psi)(s\uu)|_p\leq -\big(\epsilon S_0-e_L)T\log p$$
for $t\in\mathbb{N}^{n}$ with $|t|<T$ and for $s\in \mathbb{N}$.
\end{proposition}
\begin{proof}
By Lemma \ref{dis}
$$|(\diff^{t}\Psi)(s\w)-(\diff^{t}\Psi)(s\uu)|_p\leq |\dd^{-1}|_p^{2nT}|l(u)|_p=p^{2ne_LT}|l(u)|_p,$$
and by Proposition \ref {sch}
$$|(\diff^{t}\Psi)(s\w)|_p\leq p^{-(\epsilon S_0-e_L)T}\max\Big\{1, p^{\left((2n-1)e_L+\epsilon S_0+\frac{1}{p-1}\right)T}S_0^{S_0T}|l(u)|_p\Big\}.$$
Hence
\begin{equation*}\begin{split}
|(\diff^{t}\Psi)(s\uu)|_p&\leq \max\{|(\diff^{t}\Psi)(s\w)|_p, |(\diff^{t}\Psi)(s\w)-(\diff^{t}\Psi)(s\uu)|_p\}\\
&\leq p^{-(\epsilon S_0-e_L)T}\max\Big\{1, p^{\left((2n-1)e_L+\epsilon S_0+\frac{1}{p-1}\right)T}S_0^{S_0T}|l(u)|_p,\\
&\hspace*{5.8cm} p^{\left((2n-1)e_L+\epsilon S_0\right)T}|l(u)|_p\Big\}\\
&\leq p^{-(\epsilon S_0-e_L)T}\max\Big\{1, p^{\left((2n-1)e_L+\epsilon S_0+\frac{1}{p-1}\right)T}S_0^{S_0T}|l(u)|_p\Big\}.
\end{split}\end{equation*}
On the other hand
$$p^{\left((2n-1)e_L+\epsilon S_0+\frac{1}{p-1}\right)T}S_0^{S_0T}|l(u)|_p\leq 1$$
if and only if
$$|l(u)|_p\leq p^{-\left((2n-1)e_L+\epsilon S_0+\frac{1}{p-1}\right)T}S_0^{-S_0T}.$$
In other words, if
$$\log|l(u)|_p\leq -\left((2n-1)e_L+\epsilon S_0+\frac{1}{p-1}\right)T\log p-{S_0T}\log S_0$$
then
$$|(\diff^{t}\Psi)(s\uu)|_p\leq p^{-(\epsilon S_0-e_L)T}.$$
This means that there is a positive constant $c_4$ such that if
$$\log |l(u)|_p\leq -c_4\Big(\Big(S_0+\dfrac{1}{p-1}+e_L\Big)T\log p+S_0T\log S_0\Big)$$
then
$$\log |(\diff^{t}\Psi)(s\uu)|_p\leq -\big(\epsilon S_0-e_L)T\log p.$$
The proposition is proved.
\end{proof}

\subsection{A lower bound}

Using the Liouville's inequality, we derive the following result that will be crucial in the proof of the main result.
\begin{proposition} \label{lowerbound}
Let $s$ be an integer such that $0\leq s<S$. Assume that $\Psi$ has a zero at $s\uu$ of exact order $T'$ along $\W$ for some positive integer $T'$. Let $t$ be any element in $\mathbb Z^n_{\geq 0}$ with $|t|=T'$ such that $(\diff^t\Psi)(s\uu)\neq 0$, then
$$\log|(\diff^t\Psi)(s\uu)|_p> -c_5(T'(\cc+\log\cd+\log(D+T'\ccc))+D_0b+DS^2h)$$
for some positive constant $c_5$.
\end{proposition}
\begin{proof}
As in the proof of Proposition \ref{siegel1}, for $y\in K_v$ and for $x\in \Lambda_v^n$ we define
$$\Psi^*_s(y,x):=\Psi(y,su+x),\quad E_{s}(x):=E_0(su,x), \quad \Psi_s(y,x):=\Psi^*_s(y,x)E_{s}(x)^D.$$
By our assumption
$$0=(\diff^{\tau}\Psi)(s\uu)=(\diff^\tau\Psi)(0,su)=(\diff^{\tau}\Psi^*_s)(0,0)$$
for $\tau\in\mathbb{N}^n$ with $|\tau|<T'$.
Leibniz' rule gives
$$(\diff^{\tau}\Psi_s)(0,0)=\big(\diff^{\tau}(\Psi^*_sE_s^D)\big)(0,0)=0.$$
Using Leibniz' rule again, one gets
$$(\diff^t\Psi)(s\uu)=\big(\diff^t(\Psi_sE_s^{-D})\big)(0,0)=(\diff^{t}\Psi_{s})(0,0)E_{s}^{-D}(0).$$
The same arguments as in the proof of Proposition \ref{siegel1} (just replace $S_0$ by $S$) show that
$h((\diff^t\Psi_{s})(0,0))$
is
$$\ll T'(\cc+\log\cd+\log(D+T'\log\ccc))+D_0b+DS^2h.$$
Furthermore
$$h(E_s^{-D}(0))=h(E_0(su,0)^{-D})=Dh(E_0(su,0))\ll DS^2h.$$
Since $(\diff^{t}\Psi)(s\uu)\neq 0$ Liouville's inequality gives
\begin{equation*}\begin{split}
\log|(\diff^t\Psi)(s\uu)|_p&\gg -h((\diff^t\Psi)(s\uu))=-h\big((\diff^t\Psi_{s})(0,0)E_s(0)^{-D}\big)\\
&\gg -\big(T'(\cc+\log\cd+\log(D+T'\ccc))+D_0b+DS^2h\big)
\end{split}\end{equation*}
and the proposition follows.
\end{proof}

\subsection{Multiplicity estimates}

Another crucial point for proving the theorem is the following lemma. For the proof we use \cite{phi}, but we also refer to \cite{waa2} (and to \cite{waa3}, where the multiplicity estimates part has been published); the former mentioned result is a modification of the multiplicity estimate part of latter habilitation thesis.
\begin{lemma}\label{est}
Let $\eta$ denote the point $(0,\gamma)$ and $\Gamma(\eta)$ denote the set $\{\eta^i; i\in\mathbb{N}\}$. Let $H(\mathscr H; D_0,D)$ and $H(\mathscr G; D_0,D)$ be the Hilbert-Samuel functions associated with the ideal of $\mathscr H$ and $\mathscr G$ respectively.
If $\Psi$ vanishes at any point of the set $\{s\uu; 0\leq s<S\}$ along $\W$ of order $\geq T$, then there are a connected algebraic subgroup $\mathscr H$ defined over $K$ distinct from $\G$ and a positive constant $c_6$ satisfying that the quantity
$$\binom{T+\codim_{\W_p}\W_p\cap T_\mathscr H}{\codim_{\W_p}\W_p\cap T_\mathscr H}\card((\Gamma(\eta)+\mathscr H)/\mathscr H)H(\mathscr H; D_0,D)$$
is bounded from above by $c_6H(\mathscr G; D_0,D)$,
where $\W_p:=\W\otimes_{K_v}\mathbb C_p$ and $T_\mathscr H=\Lie(\mathscr H)\otimes_K\mathbb C_p$.
\end{lemma}
\begin{proof}
We associate with $P$ the bi-homogeneous polynomial $P^h$ in $N+2$ variables $Y_0,Y_1,X_0,\ldots,X_N$ of degree $D_0$ in $Y_0,Y_1$ and degree $D$ in $X_0,\ldots,X_N$ defined by
$$P^h(Y_0,Y_1,X_0,\ldots,X_N):=Y_0^{D_0}P(Y_1/Y_0,X_0,\ldots,X_N).$$
Since $\ord_{s\uu,\W}\Psi\geq T$ the order at any point $s\uu$ along $\W$ of the analytic function $P^h(1,y,1,f_1(x),\ldots,f_N(x))$  is at least $T$. This also means that the order of $P^h(1,y,1,f_1(x),\ldots,f_N(x))$ along $\W_p$ at any point $s\uu$ is at least $T$. Therefore the lemma follows immediately from Theorem 2.1 of \cite{phi}.
\end{proof}

\subsection{Choice of parameters and proof of Theorem \ref{mth}}\label{choice}

We choose parameters as follows. Let $c$ be a large enough positive constant and
\begin{equation*}\begin{split}
S_0&=[c\omega_L(\log b+\log h)],\\
D_0&=[c^{5n+1}S_0^{n+1}h^n],\\
S&=[c^2S_0],\\
D&=[c^{5n+1}S_0^{n}bh^{n-1}],\\
T&=[c^{5n+6}S_0^{n+1}bh^{n}],
\end{split}\end{equation*}
where $[x]$ for real $x$ is defined to be the largest integer less than or equal to $x$. Our parameters satisfy $D_0D^n\geq c_2S_0T^n$. Proposition \ref{siegel1} gives a polynomial $P$ in $N+2$ variables $Y,X_0\ldots,X_{N}$ with coefficients in $\mathcal O_K$, homogeneous in $X_0,\ldots,X_{N}$ of degree $D$, and with $\deg P_Y\leq D_0$ such that
\\
\hspace*{0.5cm} 1. $P$ does not vanish identically on $\mathscr G$,
\\
\hspace*{0.5cm} 2. $(\diff^{t}\Psi)(s\uu)=0, \forall 0\leq s<S_0, \forall t=(t_1,\ldots,t_n), 0\leq t_1,\ldots,t_n<2T$,
\\
\hspace*{0.5cm} 3. $h(P)\leq c_3(T(\cc+\log\cd+\log (D+T\ccc))+D_0b+DS_0^2h);$
\\
here we write $\Psi$ for $\Psi_P$.
\begin{lemma}
$$\log |l(u)|_p> -c_4\Big(\Big(S_0+\dfrac{1}{p-1}+e_L\Big)T\log p+S_0T\log S_0\Big).$$
\end{lemma}
\begin{proof}
On assuming that
$$\log |l(u)|_p\leq -c_4\Big(\Big(S_0+\dfrac{1}{p-1}+e_L\Big)T\log p+S_0T\log S_0\Big)$$
Proposition \ref{upperbound} gives
$$\log |(\diff^{t}\Psi)(s\uu)|_p\leq -\big(\epsilon S_0-e_L)T\log p.$$
We shall show that the order of $\Psi$ along $\W$ at any point of the set $\{s\uu; 0\leq s<S\}$ is at least $T$. Otherwise there is some point $s_0\uu$ with $0\leq s_0<S$ at which the exact order along $\W$ is $T_0<T$. This means that there exists a $n$-tuple $\tau\in\mathbb{N}^n$ such that $|\tau|=T_0$ and $(\diff^\tau\Psi)(s_0\uu)\neq 0$. We apply Proposition \ref{lowerbound} to get
\begin{equation*}\begin{split}
\log|(\diff^\tau\Psi)(s_0\uu)|_p> -c_5(T_0(\cc+\log\cd+\log (D+T_0\ccc))+D_0b+DS^2h).
\end{split}\end{equation*}
The comparison with the lower bound above implies that the quantity
\begin{equation*}\begin{split}
-\big(\epsilon S_0-e_L)T\log p
\end{split}\end{equation*}
is bounded from below by
$$-c_5(T(\cc+\log\cd+\log(D+T\ccc))+D_0b+DS^2h).$$
This implies that
\begin{equation*}\begin{split}
\big(\epsilon S_0-e_L)T\log p \leq c_5(T(\cc+\log\cd+\log(D+T\ccc))+D_0b+DS^2h)
\end{split}\end{equation*}
and shows that
\begin{equation*}\begin{split}
\Big(\dfrac{1}{3d^2}\log 2&\Big)T(S_0-e_L)
\\
&\leq c_5(T(\cc+\log\cd+\log(D+T\ccc))+D_0b+DS^2h).
\end{split}\end{equation*}
This means that there is a positive constant $c_{7}$ satisfying
$$T(S_0-e_L)\leq c_{7}(T(\cc+\log\cd+\log(D+T\ccc))+D_0b+DS^2h).$$
We get a contradiction because this cannot hold if $c$ is sufficiently large. Therefore $\Psi$ vanishes at any point of the set $\{s\uu; 0\leq s<S\}$ of order at least $T$ along $\W$. By Lemma \ref{est}, there is a connected algebraic subgroup $\mathscr H$ defined over $K$ distinct from $\G$ satisfying
$$\binom{T+\codim_{\W_p}\W_p\cap T_\mathscr H}{\codim_{\W_p}\W_p\cap T_\mathscr H}\card((\Gamma(\eta)+\mathscr H)/\mathscr H)H(\mathscr H; D_0,D)$$
is bounded form above by $c_6H(\mathscr G; D_0,D).$
Since $G$ and $\mathbb G_a$ are disjoint, there are subgroups $H_a$ of $\mathbb G_a$ and $H$ of $G$ (defined over $K$) such that $\mathscr H=H_a\times H$. Let $n_a$ be the dimension of $H_a$ and $n'$ be the dimension of $H$. We know that $H(\mathscr H; D_0,D)\gg D_0^{n_a}D^{n'}$ and $H(\mathscr G; D_0,D)\ll D_0D^{n}$. The above inequality gives
$$\binom{T+\codim_{\W_p}\W_p\cap T_\mathscr H}{\codim_{\W_p}\W_p\cap T_\mathscr H}\card((\Gamma(\eta)+\mathscr H)/\mathscr H)\ll D_0^{1-n_a}D^{n-n'}.$$
We shall show that $H$ must be the trivial group $\{e\}$. Indeed, if not, then we get a proper quotient $\pi: G\rightarrow G/H$ inducing a linear map $\pi_{*}: \frak g\rightarrow \frak g/\frak h$ of Lie algebras which maps the hyperplane $W$ onto the quotient $(W+\frak h)/\frak h$; here $\frak g$ and $\frak h$ denote the Lie algebra of $G$ and $H$ respectively. Furthermore we have $\tau(G,W)=\frac{n-1}{n}$ and since $(G,W)$ is semistable over $\overline{\mathbb Q}$, it is also semistable over $K$. This gives
\begin{equation*}\begin{split}
\tau(G,W)&\leq \tau(G/H, \pi_*(W))\\
&=\dfrac{\dim(W+\frak h)-\dim \frak h}{\dim G-\dim H}\\
&=\dfrac{\dim(W+\frak h)-n'}{n-n'}.
\end{split}\end{equation*}
But
$$n-1=\dim W\leq \dim(W+\frak h)\leq n$$
and this shows that $\dim(W+\frak h)$ must be $n$, i.e. $\dim(\W_p+T_\HH)=n$. This gives
$$\codim_{\W_p} \W_p\cap T _\HH=\dim(\W_p+T_\HH)-\dim T_\HH=n+1-n_a-n',$$
and shows that
$$\binom{T+n+1-n_a-n'}{n+1-n_a-n'}\ll D_0^{1-n_a}D^{n-n'}.$$
We deduce that
$$T^{n+1-n_a-n'}\leq c_{8}D_0^{1-n_a}D^{n-n'},$$
for some positive constant $c_{8}$ and get a contradiction to $T>cD_0,cD$. As a consequence we obtain $H=\{e\}$, and therefore $T_\HH\cap \W_p$ must be trivial. One gets
$$\codim_{\W_p} \W_p\cap T _\HH=\dim \W_p=n.$$
Moreover, $\Gamma(\gamma)\cap\HH$ must also be trivial and hence
$$\card((\Gamma(\gamma)+\mathscr H)/\mathscr H)=\card \Gamma(\eta)=S.$$
We obtain
$$\binom{T+n}{n}S\ll D_0^{1-n_a}D^n\leq D_0D^n.$$
This therefore shows that $T^nS\leq c_{9}D_0D^n$
for some positive constant $c_{9}$, and again gives a contradiction because of the choice of parameters. The lemma is proved.
\end{proof}
In order to finish the proof of the theorem, we use the above lemma and the fact that $\log r_p^{-1}=\frac{\log p}{p-1}< 2$
to get
\begin{equation*}\begin{split}\log|l(u)|_p&>-c_{10}(S_0T\log p+S_0T\log S_0+Te_L\log p)\\
&>-c_{11}\big(S_0^{n+2}bh^n\log p+S_0^{n+2}(\log S_0)bh^n\big)\\
&>-c_{12}S_0^{n+3}bh^n\log p\end{split}\end{equation*}
for some positive constants $c_{10}, c_{11}$ and $c_{12}$.
In other words there is a positive constant $c_0$ independent of $b,h,p$ such that
\begin{equation*}\begin{split}
\log|l(u)|_p>-c_0\omega_L^{n+3}bh^n(\log b+\log h)^{n+3}\log p.
\end{split}\end{equation*}
The first assertion of the theorem is therefore proved and this together with Section 2.2 completes the proof of the theorem.

\section{Acknowledgments}

The authors are most grateful to Professor G. W\"ustholz for very insightful discussions and for constant encouragement to work on the topic. Moreover, the authors thank an anonymous referee for his careful reading of the manuscript. The first author was supported by Austrian Science Fund (FWF): P24574. The second author was supported by grant PDFMP2\_122850 funded by the Swiss National Science Foundation (SNSF).

\end{document}